\documentclass[11pt]{amsart}
\usepackage{amsmath}
\usepackage{amsfonts}
\usepackage{eucal}
\usepackage{oldgerm}
\usepackage{latexsym}
\usepackage{amsmath}
\usepackage{amscd}
\usepackage{amssymb} % \mathfrak , \mathbb
\usepackage{amsmath}
\usepackage{amsthm}
\usepackage{latexsym}
\usepackage{arydshln}

\usepackage{oldgerm}
\usepackage{latexsym}
\usepackage{amsmath}
\usepackage{amscd}
\usepackage{amssymb} % \mathfrak , \mathbb
\usepackage{amsmath}
\usepackage{amsthm}
\usepackage{latexsym}
\usepackage{arydshln}
\makeatletter
\@namedef{subjclassname@2020}{%
  \textup{2020} Mathematics Subject Classification}
\makeatother

\usepackage{xcolor}

\begin{document}

\title [Period of the Ikeda type lift]{Period of the Ikeda type lift for $E_{7,3}$}

\author{Hidenori Katsurada, Henry H. Kim and Takuya Yamauchi}
\date{May 24, 2022}
\thanks{The first author is partially supported by JSPS KAKENHI Grant Number (B) No.16H03919. The second author is partially supported by NSERC grant \#482564. The third author is partially supported by 
 JSPS KAKENHI Grant Number (B) No.19H01778.}

\address{Hidenori Katsurada, Department of Mathematics, Hokkaido University, Kita 10, Nishi 8, Kita-Ku, Sapporo, Hokkaido, 060-0810, Japan \\ and 
Muroran Institute of Technology\\
27-1 Mizumoto Muroran 050\\ Japan}
\email{hidenori@mmm.muroran-it.ac.jp}
\address{Henry H. Kim \\
Department of mathematics \\
 University of Toronto \\
Toronto, Ontario M5S 2E4, CANADA \\
and Korea Institute for Advanced Study, Seoul, KOREA}
\email{henrykim@math.toronto.edu}
\address{Takuya Yamauchi \\
Mathematical Inst. Tohoku Univ.\\
 6-3, Aoba, Aramaki, Aoba-Ku, Sendai 980-8578, JAPAN}
\email{takuya.yamauchi.c3@tohoku.ac.jp}

\maketitle

%Greek letters

\newcommand{\alp}{\alpha}
\newcommand{\bet}{\beta}
\newcommand{\gam}{\gamma}
\newcommand{\del}{\delta}
\newcommand{\eps}{\epsilon}
\newcommand{\zet}{\zeta}
\newcommand{\tht}{\theta}
\newcommand{\iot}{\iota}
\newcommand{\kap}{\kappa}
\newcommand{\lam}{\lambda}
\newcommand{\sig}{\sigma}
\newcommand{\ups}{\upsilon}
\newcommand{\ome}{\omega}
\newcommand{\vep}{\varepsilon}
\newcommand{\vth}{\vartheta}
\newcommand{\vpi}{\varpi}
\newcommand{\vrh}{\varrho}
\newcommand{\vsi}{\varsigma}
\newcommand{\vph}{\varphi}
\newcommand{\Gam}{\Gamma}
\newcommand{\Del}{\Delta}
\newcommand{\Tht}{\Theta}
\newcommand{\Lam}{\Lambda}
\newcommand{\Sig}{\Sigma}
\newcommand{\Ups}{\Upsilon}
\newcommand{\Ome}{\Omega}

%fraktur letters

\newcommand{\frka}{{\mathfrak a}}    \newcommand{\frkA}{{\mathfrak A}}
\newcommand{\frkb}{{\mathfrak b}}    \newcommand{\frkB}{{\mathfrak B}}
\newcommand{\frkc}{{\mathfrak c}}    \newcommand{\frkC}{{\mathfrak C}}
\newcommand{\frkd}{{\mathfrak d}}    \newcommand{\frkD}{{\mathfrak D}}
\newcommand{\frke}{{\mathfrak e}}    \newcommand{\frkE}{{\mathfrak E}}
\newcommand{\frkf}{{\mathfrak f}}    \newcommand{\frkF}{{\mathfrak F}}
\newcommand{\frkg}{{\mathfrak g}}    \newcommand{\frkG}{{\mathfrak G}}
\newcommand{\frkh}{{\mathfrak h}}    \newcommand{\frkH}{{\mathfrak H}}
\newcommand{\frki}{{\mathfrak i}}    \newcommand{\frkI}{{\mathfrak I}}
\newcommand{\frkj}{{\mathfrak j}}    \newcommand{\frkJ}{{\mathfrak J}}
\newcommand{\frkk}{{\mathfrak k}}    \newcommand{\frkK}{{\mathfrak K}}
\newcommand{\frkl}{{\mathfrak l}}    \newcommand{\frkL}{{\mathfrak L}}
\newcommand{\frkm}{{\mathfrak m}}    \newcommand{\frkM}{{\mathfrak M}}
\newcommand{\frkn}{{\mathfrak n}}    \newcommand{\frkN}{{\mathfrak N}}
\newcommand{\frko}{{\mathfrak o}}    \newcommand{\frkO}{{\mathfrak O}}
\newcommand{\frkp}{{\mathfrak p}}    \newcommand{\frkP}{{\mathfrak P}}
\newcommand{\frkq}{{\mathfrak q}}    \newcommand{\frkQ}{{\mathfrak Q}}
\newcommand{\frkr}{{\mathfrak r}}    \newcommand{\frkR}{{\mathfrak R}}
\newcommand{\frks}{{\mathfrak s}}    \newcommand{\frkS}{{\mathfrak S}}
\newcommand{\frkt}{{\mathfrak t}}    \newcommand{\frkT}{{\mathfrak T}}
\newcommand{\frku}{{\mathfrak u}}    \newcommand{\frkU}{{\mathfrak U}}
\newcommand{\frkv}{{\mathfrak v}}    \newcommand{\frkV}{{\mathfrak V}}
\newcommand{\frkw}{{\mathfrak w}}    \newcommand{\frkW}{{\mathfrak W}}
\newcommand{\frkx}{{\mathfrak x}}    \newcommand{\frkX}{{\mathfrak X}}
\newcommand{\frky}{{\mathfrak y}}    \newcommand{\frkY}{{\mathfrak Y}}
\newcommand{\frkz}{{\mathfrak z}}    \newcommand{\frkZ}{{\mathfrak Z}}

%math boldface latters

\newcommand{\bfa}{{\mathbf a}}    \newcommand{\bfA}{{\mathbf A}}
\newcommand{\bfb}{{\mathbf b}}    \newcommand{\bfB}{{\mathbf B}}
\newcommand{\bfc}{{\mathbf c}}    \newcommand{\bfC}{{\mathbf C}}
\newcommand{\bfd}{{\mathbf d}}    \newcommand{\bfD}{{\mathbf D}}
\newcommand{\bfe}{{\mathbf e}}    \newcommand{\bfE}{{\mathbf E}}
\newcommand{\bff}{{\mathbf f}}    \newcommand{\bfF}{{\mathbf F}}
\newcommand{\bfg}{{\mathbf g}}    \newcommand{\bfG}{{\mathbf G}}
\newcommand{\bfh}{{\mathbf h}}    \newcommand{\bfH}{{\mathbf H}}
\newcommand{\bfi}{{\mathbf i}}    \newcommand{\bfI}{{\mathbf I}}
\newcommand{\bfj}{{\mathbf j}}    \newcommand{\bfJ}{{\mathbf J}}
\newcommand{\bfk}{{\mathbf k}}    \newcommand{\bfK}{{\mathbf K}}
\newcommand{\bfl}{{\mathbf l}}    \newcommand{\bfL}{{\mathbf L}}
\newcommand{\bfm}{{\mathbf m}}    \newcommand{\bfM}{{\mathbf M}}
\newcommand{\bfn}{{\mathbf n}}    \newcommand{\bfN}{{\mathbf N}}
\newcommand{\bfo}{{\mathbf o}}    \newcommand{\bfO}{{\mathbf O}}
\newcommand{\bfp}{{\mathbf p}}    \newcommand{\bfP}{{\mathbf P}}
\newcommand{\bfq}{{\mathbf q}}    \newcommand{\bfQ}{{\mathbf Q}}
\newcommand{\bfr}{{\mathbf r}}    \newcommand{\bfR}{{\mathbf R}}
\newcommand{\bfs}{{\mathbf s}}    \newcommand{\bfS}{{\mathbf S}}
\newcommand{\bft}{{\mathbf t}}    \newcommand{\bfT}{{\mathbf T}}
\newcommand{\bfu}{{\mathbf u}}    \newcommand{\bfU}{{\mathbf U}}
\newcommand{\bfv}{{\mathbf v}}    \newcommand{\bfV}{{\mathbf V}}
\newcommand{\bfw}{{\mathbf w}}    \newcommand{\bfW}{{\mathbf W}}
\newcommand{\bfx}{{\mathbf x}}    \newcommand{\bfX}{{\mathbf X}}
\newcommand{\bfy}{{\mathbf y}}    \newcommand{\bfY}{{\mathbf Y}}
\newcommand{\bfz}{{\mathbf z}}    \newcommand{\bfZ}{{\mathbf Z}}

%caligraphic letters

\newcommand{\cala}{{\mathcal A}}
\newcommand{\calb}{{\mathcal B}}
\newcommand{\calc}{{\mathcal C}}
\newcommand{\cald}{{\mathcal D}}
\newcommand{\cale}{{\mathcal E}}
\newcommand{\calf}{{\mathcal F}}
\newcommand{\calg}{{\mathcal G}}
\newcommand{\calh}{{\mathcal H}}
\newcommand{\cali}{{\mathcal I}}
\newcommand{\calj}{{\mathcal J}}
\newcommand{\calk}{{\mathcal K}}
\newcommand{\call}{{\mathcal L}}
\newcommand{\calm}{{\mathcal M}}
\newcommand{\caln}{{\mathcal N}}
\newcommand{\calo}{{\mathcal O}}
\newcommand{\calp}{{\mathcal P}}
\newcommand{\calq}{{\mathcal Q}}
\newcommand{\calr}{{\mathcal R}}
\newcommand{\cals}{{\mathcal S}}
\newcommand{\calt}{{\mathcal T}}
\newcommand{\calu}{{\mathcal U}}
\newcommand{\calv}{{\mathcal V}}
\newcommand{\calw}{{\mathcal W}}
\newcommand{\calx}{{\mathcal X}}
\newcommand{\caly}{{\mathcal Y}}
\newcommand{\calz}{{\mathcal Z}}

%math script

\newcommand{\scra}{{\mathscr A}}
\newcommand{\scrb}{{\mathscr B}}
\newcommand{\scrc}{{\mathscr C}}
\newcommand{\scrd}{{\mathscr D}}
\newcommand{\scre}{{\mathscr E}}
\newcommand{\scrf}{{\mathscr F}}
\newcommand{\scrg}{{\mathscr G}}
\newcommand{\scrh}{{\mathscr H}}
\newcommand{\scri}{{\mathscr I}}
\newcommand{\scrj}{{\mathscr J}}
\newcommand{\scrk}{{\mathscr K}}
\newcommand{\scrl}{{\mathscr L}}
\newcommand{\scrm}{{\mathscr M}}
\newcommand{\scrn}{{\mathscr N}}
\newcommand{\scro}{{\mathscr O}}
\newcommand{\scrp}{{\mathscr P}}
\newcommand{\scrq}{{\mathscr Q}}
\newcommand{\scrr}{{\mathscr R}}
\newcommand{\scrs}{{\mathscr S}}
\newcommand{\scrt}{{\mathscr T}}
\newcommand{\scru}{{\mathscr U}}
\newcommand{\scrv}{{\mathscr V}}
\newcommand{\scrw}{{\mathscr W}}
\newcommand{\scrx}{{\mathscr X}}
\newcommand{\scry}{{\mathscr Y}}
\newcommand{\scrz}{{\mathscr Z}}

%math Bbb

\newcommand{\AAA}{{\mathbb A}} %not \AA
\newcommand{\BB}{{\mathbb B}}
\newcommand{\CC}{{\mathbb C}}
\newcommand{\DD}{{\mathbb D}}
\newcommand{\EE}{{\mathbb E}}
\newcommand{\FF}{{\mathbb F}}
\newcommand{\GG}{{\mathbb G}}
\newcommand{\HH}{{\mathbb H}}
\newcommand{\II}{{\mathbb I}}
\newcommand{\JJ}{{\mathbb J}}
\newcommand{\KK}{{\mathbb K}}
\newcommand{\LL}{{\mathbb L}}
\newcommand{\MM}{{\mathbb M}}
\newcommand{\NN}{{\mathbb N}}
\newcommand{\OO}{{\mathbb O}}
\newcommand{\PP}{{\mathbb P}}
\newcommand{\QQ}{{\mathbb Q}}
\newcommand{\RR}{{\mathbb R}}
\newcommand{\SSS}{{\mathbb S}} %not \SS
\newcommand{\TT}{{\mathbb T}}
\newcommand{\UU}{{\mathbb U}}
\newcommand{\VV}{{\mathbb V}}
\newcommand{\WW}{{\mathbb W}}
\newcommand{\XX}{{\mathbb X}}
\newcommand{\YY}{{\mathbb Y}}
\newcommand{\ZZ}{{\mathbb Z}}

%typewriter

\newcommand{\tta}{\hbox{\tt a}}    \newcommand{\ttA}{\hbox{\tt A}}
\newcommand{\ttb}{\hbox{\tt b}}    \newcommand{\ttB}{\hbox{\tt B}}
\newcommand{\ttc}{\hbox{\tt c}}    \newcommand{\ttC}{\hbox{\tt C}}
\newcommand{\ttd}{\hbox{\tt d}}    \newcommand{\ttD}{\hbox{\tt D}}
\newcommand{\tte}{\hbox{\tt e}}    \newcommand{\ttE}{\hbox{\tt E}}
\newcommand{\ttf}{\hbox{\tt f}}    \newcommand{\ttF}{\hbox{\tt F}}
\newcommand{\ttg}{\hbox{\tt g}}    \newcommand{\ttG}{\hbox{\tt G}}
\newcommand{\tth}{\hbox{\tt h}}    \newcommand{\ttH}{\hbox{\tt H}}
\newcommand{\tti}{\hbox{\tt i}}    \newcommand{\ttI}{\hbox{\tt I}}
\newcommand{\ttj}{\hbox{\tt j}}    \newcommand{\ttJ}{\hbox{\tt J}}
\newcommand{\ttk}{\hbox{\tt k}}    \newcommand{\ttK}{\hbox{\tt K}}
\newcommand{\ttl}{\hbox{\tt l}}    \newcommand{\ttL}{\hbox{\tt L}}
\newcommand{\ttm}{\hbox{\tt m}}    \newcommand{\ttM}{\hbox{\tt M}}
\newcommand{\ttn}{\hbox{\tt n}}    \newcommand{\ttN}{\hbox{\tt N}}
\newcommand{\tto}{\hbox{\tt o}}    \newcommand{\ttO}{\hbox{\tt O}}
\newcommand{\ttp}{\hbox{\tt p}}    \newcommand{\ttP}{\hbox{\tt P}}
\newcommand{\ttq}{\hbox{\tt q}}    \newcommand{\ttQ}{\hbox{\tt Q}}
\newcommand{\ttr}{\hbox{\tt r}}    \newcommand{\ttR}{\hbox{\tt R}}
\newcommand{\tts}{\hbox{\tt s}}    \newcommand{\ttS}{\hbox{\tt S}}
\newcommand{\ttt}{\hbox{\tt t}}    \newcommand{\ttT}{\hbox{\tt T}}
\newcommand{\ttu}{\hbox{\tt u}}    \newcommand{\ttU}{\hbox{\tt U}}
\newcommand{\ttv}{\hbox{\tt v}}    \newcommand{\ttV}{\hbox{\tt V}}
\newcommand{\ttw}{\hbox{\tt w}}    \newcommand{\ttW}{\hbox{\tt W}}
\newcommand{\ttx}{\hbox{\tt x}}    \newcommand{\ttX}{\hbox{\tt X}}
\newcommand{\tty}{\hbox{\tt y}}    \newcommand{\ttY}{\hbox{\tt Y}}
\newcommand{\ttz}{\hbox{\tt z}}    \newcommand{\ttZ}{\hbox{\tt Z}}

\makeatletter 
\newcommand{\LEQQ}{\mathrel{\mathpalette\gl@align<}}
\newcommand{\GEQQ}{\mathrel{\mathpalette\gl@align>}}
\newcommand{\gl@align}[2]{\lower.6ex\vbox{\baselineskip\z@skip\lineskip\z@ 
\ialign{$\m@th#1\hfil##\hfil$\crcr#2\crcr=\crcr}}}
\makeatother

\newcommand{\ua}{\underline{a}}
\newcommand{\GK}{\mathrm{GK}}
\newcommand{\ord}{\mathrm{ord}}
\def\mattwo(#1;#2;#3;#4){\left(\begin{matrix}
                               #1 & #2 \\
                               #3  & #4
                                      \end{matrix}\right)}
\theoremstyle{plain}
\newtheorem{theorem}{Theorem}[section]
\newtheorem{lemma}[theorem]{Lemma}
\newtheorem{proposition}[theorem]{Proposition}
\newtheorem{definition}[theorem]{Definition}
\newtheorem{conjecture}{Conjecture}
\newtheorem{comment}{Comment}[section]
\theoremstyle{remark}
\newtheorem{remark}[theorem]{\bf {Remark}}
\newtheorem{corollary}[theorem]{\bf {Corollary}}
\newtheorem{formula}{\bf {Formula}}[section]
\newtheorem{question}{\bf {Question}}[section]
\numberwithin{equation}{section}

%\tableofcontents

\begin{abstract}
In \cite{K-Y}, the second and the third named authors constructed the Ikeda type lift for the exceptional group $E_{7,3}$ from an elliptic modular cusp form. In this paper, we prove an explicit formula for the period or the Petersson norm of the Ikeda type lift 
 in terms of the product of the special values of the symmetric square $L$-function of the elliptic modular form.
 There are similar works done by 
the first author with his collaborator, but new technical inputs are required and developed to overcome some difficulties coming from the 
hugeness of $E_{7,3}$.   
\end{abstract}

\section{Introduction}

It is an interesting and important problem to show the algebraicity of ratios of the period (or the Petersson norm) of an elliptic modular form and that of its lift. For example, let $f$ be a Hecke eigen cusp form of weight $2k$ with respect to $SL_2(\Bbb Z)$, and let $I_f$ be the Saito-Kurokawa lift of $f$, which is a Hecke eigen cusp form on $Sp_4$ (rank 2) of weight $k+1$ (even). Then Furusawa \cite{Fu} showed that 
$\dfrac {\langle I_f,I_f\rangle}{\langle f,f\rangle}\in K_f$, where $K_f$ is the Hecke field of $f$, i.e., the number field obtained from $\Bbb Q$ by adjoining all Fourier coefficients of $f$. Ikeda \cite{Ik} constructed the Hecke cusp form $I_f$ on $Sp_{4n}$ (rank $2n$) of weight $k+n$ (even), called the Duke-Imamoglu-Ikeda lift of $f$. When $n=1$, it is a Saito-Kurokawa lift. In this case, Choie and Kohnen \cite{CK} showed that $\dfrac {\langle I_f,I_f\rangle}{\langle f,f\rangle^n}\in K_f$. The first named author and Kawamura \cite{KK} gave its refinement: More precisely, 
let $h$ be the Hecke eigenform of weight $k+\frac 12$ for $\Gamma_0(4)$, corresponding to $f$ under the Shimura correspondence. Let $\pi_f$ be the cuspidal representation of $GL_2$ attached to $f$, and $L(s,{\rm Sym}^2\pi_f)$ be the symmetric square $L$-function.
Then they proved a more refined equality
$$\frac {\langle I_f,I_f\rangle}{\langle h,h\rangle}=2^{-(2n-3)k-2n+1} \Lambda(k+n,f)\tilde\xi(2n)\prod_{i=1}^{n-1} \tilde\Lambda(2i+1,{\rm Sym}^2\pi_f)\tilde\xi(2i),
$$
where $\tilde\Lambda(s, {\rm Sym}^2\pi_f)=4(2\pi)^{-2s}\Gamma(s)\Gamma(s+k-1) L(s, {\rm Sym}^2\pi_f)$, $\Lambda(s,f)=2(2\pi)^{-s}\Gamma(s)L(s,f)$, and $\tilde\xi(s)=2(2\pi)^{-s}\Gamma(s)\zeta(s)$. The period relation of Choie and Kohnen is an easy consequence of this equality. We note that there are some other results concerning the periods of the Hermitian Ikeda lifts (cf. \cite{Kat2}, \cite{Kat3}).

Let ${\bf G}$ be a connected reductive group of type $E_{7,3}$ (cf. \cite{Bai}, \cite[Section 3]{K-Y}). 
Let $k\geq 10$ be a positive integer. For a primitive form $f\in S_{2k-8}(SL_2(\ZZ))$, the second and the third named authors \cite{K-Y} constructed a Hecke eigen cusp form $F_f$ of weight $2k$ for ${\bf G}(\ZZ)$, which is an analogue of the Ikeda lift of an elliptic modular cusp form to 
a Siegel cusp form, and we call it the Ikeda type lift for $E_{7,3}$. (See Section 5 for the details.)

In this paper, we express the period or the Petersson norm $\langle F_f,F_f \rangle$ of $F_f$ in terms of the product of the special values of the symmetric square $L$-function of $\pi_f$:

\begin{theorem} \label{th.period-relation}
We have
\[\langle F_f, F_f \rangle =\gamma_{k}\pi^{-6k-3}\prod_{i=1}^3 L(4i-3,{\rm Sym}^2\pi_f),\]
where $\displaystyle\gamma_k={ (2k-1)!(2k-5)!(2k-9)! \ 691 \over 2^{12k-7} \cdot 3^3 \cdot 5 \cdot 7^2\cdot 13}\in \Bbb Q.$
\end{theorem}

By \cite[p. 115]{Z}, if $1\leq r\leq 2k-9$ is odd, $\dfrac {L(r, {\rm Sym}^2\pi_f)}{\langle f,f\rangle \pi^{2r+2k-9}}\in K_f$ since $f$ has weight $2k-8$. Hence we have
\begin{corollary} 
\label{cor.algebraicity-of-period-relation} $\dfrac {\langle F_f, F_f \rangle}{\langle f,f\rangle^3} \in K_f.$
\end{corollary}

The method of the proof of the main result is similar to those in \cite{KK}, \cite{Kat2} and \cite{Kat3}.
However, we have to overcome some difficulty in treating exceptional groups which we will explain below.

This paper is organized as follows. In Section 2, we fix notations on Cayley numbers $\frkC_{\QQ}$ and exceptional Jordan algebras $\frkJ_{\QQ}$ and review their properties. Moreover we define the group scheme $\calm'$ over $\ZZ$ of type $E_{6,2}$ and for $T \in \frkJ(\ZZ)_{>0}$, we define the group scheme $\calu_T$ of type $F_4$. In Section 3, we give the mass formula for $\calu_T$. The mass $\mathrm{Mass}(T)$ is defined as
\[\mathrm{Mass}(T)=\sum_{T'} {1 \over \#\calu_{T'}(\ZZ)},\]
where $T'$ runs over all $\calm'(\ZZ)$-equivalence classes of elements of $\frkJ(\ZZ)$ belonging to the genus of $T$. The mass formula, roughly speaking, expresses $\mathrm{Mass}(T)$ as an infinite product of the inverse of the local densities $\{\beta_p(T)\}_p$. Here $\beta_p(T)$ is usually defined as a limit of the number of solutions of a certain Diophantine equation over $\ZZ_p/p^n\ZZ_p$. 
(See the remark after Lemma \ref{lem.counting-principle}.) However it is not so easy to give its explicit formula unlike the case of quadratic forms or Hermitian forms. To overcome this difficulty, we define $\beta_p(T)$ in another way. (See Definition \ref{def.local-density}.) This enables us to compute $\beta_p(T)$ in Section 6. Accordingly, we formulate the mass formula following Sato \cite{Sa}. (See Theorem \ref{th.mass-formula2}.)
In Section 4, we investigate analytic properties of the Rankin-Selberg series $R(s,F,F)$ of a cusp form $F$ of weight $k$ for ${\bf G}(\ZZ)$: In particular it has a simple pole at $s=k$ (rightmost end pole) and we can express the residue at $s=k$ in terms of the period of $F$. 
In Section 5, we review the Ikeda type lift $F_f$ for $E_{7,3}$ of a primitive form $f$ for $SL_2(\ZZ)$. By construction, the Fourier coefficients of $F_f$ are expressed in terms of a product of the local Siegel series. Therefore, using the mass formula in Section 3, we express the Rankin-Selberg series $R(s,F_f,F_f)$ 
as an Euler product (Theorem \ref{th.local-global-RS2}):
\[R(s,F_f,F_f)= c\prod_p H_p(\alpha_p,p^{-s+2k}),\]
where $\alpha_p$ is the $p$-th Satake parameter for $f$, and $H_p(X,t)$ is a certain power series involving the Siegel series and the local density, and $c$ is a non-zero constant in Theorem \ref{th.mass-formula2}. In Section 6, we give an explicit formula of the local density, and show that
$c={5!\cdot7! \cdot11! \over (2\pi)^{28}}$ (Proposition 6.7).
 It is remarkable that $R(s,F_f,F_f)$ has an Euler product even though $R(s,F,F)$ does not have an Euler product for a general cusp form $F$.
In Section 7, we rewrite the formula of the Siegel series due to Karel \cite{Kar3}, and 
combined with the result in Section 6, 
we give an explicit formula of $H_p(X,t)$. Here we use Mathematica since the expression is very complicated.
Finally, combined with the residue formula of the Rankin-Selberg series in Section 4, we prove our main result in Section 8.

\smallskip

\noindent\textbf{Acknowledgments.} We thank Fumihiro Sato for many valuable comments. We also thank the referee who read our paper carefully and pointed out several errors in it.

\smallskip

\noindent {\bf Notation.} In addition to the standard symbols $\ZZ,\QQ,\RR$ and $\CC$, for a prime number $p$, let $\QQ_p$ and $\ZZ_p$ be the field of $p$-adic numbers and the ring of $p$-adic integers. Let $\Bbb A$ be the ring of adeles.

Let $\sim$ be an equivalence relation on a set $\cals$. We denote by $\cals/\sim$ the set of equivalence classes of $\cals$ under $\sim$. We use the same symbol $\cals/\sim$ to denote a complete set of representatives.
Let $G$ be a group acting on a set $\cals$. For two elements $a_1$ and $a_2$, we write $a_2 \sim_G a_1$ if $a_2=g\cdot a_1$ with $g \in G$. The relation $\sim_G$ is an equivalence relation on $\cals$ and we write $\cals/G$ instead of $\cals/\sim_G$. For square matrices $A_1,\ldots,A_r$, we write $A_1 \bot \cdots\bot A_r=\begin{pmatrix}A_1 & O       &  O \\
                                    O   & \ddots & O        \\
                                    O   &    O       &   A_r
\end{pmatrix}$. We sometimes write $\mathrm{diag}(A_1,\cdots,A_r)$ instead of $A_1 \bot \cdots \bot A_r$.

\section{Exceptional Jordan algebras and the exceptional domain}
The  Cayley numbers $\frkC_\QQ$ over $\QQ$ is an eight-dimensional vector space over $\QQ$ with basis $\{e_0=1,e_1,e_2,e_3,e_4,e_5,e_6,e_7,e_8\}$ satisfying the following rules for multiplication:
\begin{itemize}
\item[(1)] $xe_0=e_0x=x \text{ for all } x \in \frkC_{\QQ}$;
\item[(2)] $e_i^2=-e_0 \text{ for } i=1,\ldots,7$;
\item[(3)] $e_i(e_{i+1}e_{i+3})=(e_ie_{i+1})e_{i+2}=-e_0 \text{ for any } i \text{ mod } 7$.
\end{itemize}
For each $x=\sum_{i=0}^7 x_ie_i$ with $x_i \in \QQ$, the mapping $x \longrightarrow \bar x:=x_0e_0-\sum_{i=1}^7 x_ie_i$ defines an anti-involution of $\frkC_{\QQ}$. The trace and the norm on $\frkC_{\QQ}$ are defined by
\[\mathrm{Tr}(x)=x+\bar x \text{ and } N(x)=x\bar x.\]
Let $\frko \subset \frkC_{\QQ}$ be the space of integral Cayley numbers. It is a $\ZZ$-submodule of $\frkC_{\QQ}$ given by the following basis:
\begin{align*}
&\alpha_0=e_0, \quad \alpha_1=e_1, \quad  \alpha_2=e_2, \quad \alpha_3=-e_4 \\
&\alpha_4={1 \over 2}(e_1+e_2+e_3-e_4), \quad \alpha_5={1 \over 2}(-e_0-e_1-e_4+e_5)\\
&\alpha_6={1 \over 2}(-e_0+e_1-e_2+e_6), \quad \alpha_7={1 \over 2}(-e_0+e_2+e_4+e_7).
\end{align*}
For a commutative algebra $R$, we put ${\frkC}(R)=\frko \otimes_{\ZZ} R$. We note that $\frkC({\QQ})=\frkC_{\QQ}$. 
 Let $\frkJ_{\QQ}$ be the exceptional Jordan algebra consisting of matrices
\[X=(x_{ij})_{1 \le i,j \le 3}=\begin{pmatrix} a & x & y \\ \bar x & b & z \\ \bar y & \bar z & c \end{pmatrix}\]
with $a,b,c \in \QQ$ and $x,y,z \in \frkC_{\QQ}$. 
For $X_1,X_2 \in \frkJ_{\QQ}$ we define $X_1 \circ X_2$ by
\[X_1 \circ X_2 ={1 \over 2} (X_1 X_2+X_2 X_1),\]
where $X_1X_2$ and $X_2X_1$ are the usual matrix multiplications. We define an inner product on $\frkJ_{\QQ} \times \frkJ_{\QQ}$ by $(X_1,X_2):=\mathrm{Tr}(X_1 \circ X_2)$. Moreover, for $X_1=\begin{pmatrix} a_1 & x_1 & y_1 \\ \bar x_1 & b_1 & z_1 \\ \bar y_1 & \bar z_1 & c_1 \end{pmatrix}$ and $X_2=\begin{pmatrix} a_2 & x_2 & y_2 \\ \bar x_2 & b_2 & z_2 \\ \bar y_2 & \bar z_2 & c_2 \end{pmatrix}$, define
$X_1 \times X_2$ by
\[X_1 \times X_2=\begin{pmatrix} {b_1c_2+c_1b_2 -(\bar z_1z_2+z_2\bar z_1) \over 2} & A & B \\
\bar A & {a_1c_2+c_1a_2-(\bar y_1y_2+y_2\bar y_1) \over 2} & C \\
\bar B & \bar C & {a_1b_2+b_1a_2-(\bar x_1x_2+x_2\bar x_1) \over 2}
\end{pmatrix},\]
where
$\displaystyle A={-c_1x_2-c_2x_1+y_1\bar z_2+y_2\bar z_1 \over 2}, B={-b_1y_2-b_2y_1+x_1z_2+x_2z_1 \over 2}$ and \\
$\displaystyle C={-a_1z_2-a_2z_1 +\bar x_1y_2+\bar x_2y_1 \over 2}$.
Then  $X_1 \circ X_2$ and $X_1 \times X_2$ belongs to $\frkJ_{\QQ}$, and $(X_1,X_2) \in \QQ$.
We define the determinant $\det X$ and the trace $\mathrm{Tr}(X)$ by
\[\det X=abc-aN(z)-bN(y)-cN(x)+\mathrm{Tr}((xz)\bar y), \quad  \mathrm{Tr}(x)=a+b+c.\]
Then $\det X$ and $\mathrm{Tr}(x)$ belongs to $\QQ$.
 We define a lattice $\frkJ$ of $\frkJ_{\QQ}$  by 
\[\frkJ=\{ X =(x_{ij}) \in \frkJ_{\QQ} \ | \ x_{ii} \in \ZZ \text{ and } x_{ij} \in \frko \text{ for } i\not=j\},\]
 and for a commutative algebra $R$, let $\frkJ(R)=\frkJ \otimes_{\ZZ} R$. 
Then an element $X$ of $\frkJ(R)$ is expressed as
\[X=\begin{pmatrix} a & x & y \\ \bar x & b & z \\ \bar y & \bar z & c \end{pmatrix}\]
with $a,b,c \in R$, and $x,y,z \in \frkC(R)$. We note that $\frkJ(\QQ)=\frkJ_\QQ$ and $\frkJ(\ZZ)=\frkJ$.
If $R$ is not a field, $X_1 \circ X_2$ and $X_1 \times X_2$ do not necessarily belong to $\frkJ(R)$ for $X_1,X_2 \in \frkJ(R)$, but $X_1 \times X_1 \in \frkJ(R)$, and  $(X_1,X_2) \in R$.
We define
\[\frkJ(R)^{\rm ns}=\{ X \in \frkJ(R) \ | \ \det (X) \not=0\},\]
and 
\[R_3^+(R)=\{ X^2 \ | \ X \in \frkJ(R)^{\rm ns}\}.\]
It is known that if $R$ is the field $\RR$ of real numbers, $R_3(\RR)^+$ is an open convex cone in $\frkJ(\RR)$. We denote by $\overline{R_3^+(\RR)}$ the closure of $R_3^+(\RR)$ in $\frkJ(\RR) \simeq \RR^{27}$. For a subring $A$ of $\RR$ set
\[\frkJ(A)_{>0}=\frkJ(A) \cap R_3^+(\RR) \text{ and } \frkJ(A)_{\ge 0}=\frkJ(A) \cap \overline{R_3^+(\RR)}.\]

We define the exceptional domain as follows:
$$\frak T:=\{Z=X+Y\sqrt{-1}\in \frak J_\Bbb C\ |\ X,Y\in \frak J_\Bbb R,\ Y\in R^+_3(\Bbb R)\}
$$
which is a complex analytic subspace of $\Bbb C^{27}$.

\section{Mass formula for the exceptional group of type $F_4$}

For a commutative algebra $R$, we denote by  $GL(\frkJ(R))$ the group of $R$-linear transformations  of $\frkJ(R)$.  
We sometimes denote by $g\cdot X$ the action of $g \in GL(\frkJ(R))$ on $X \in \frkJ(R)$. 
Define the group schemes $\calm$ and $\calm'$ over $\ZZ$ by 
\[\calm(R)=\{ g \in GL(\frkJ(R)) \ | \ \det (g\cdot X)=\nu(g)\det X \text{ with } \nu(g) \in R^\times\}
\]
and 
\[\calm'(R)=\{ g \in \calm(R) \ |\,  \nu(g) =1 \}.
\]
We note that our definition is the same as in \cite{Gr} and also in \cite{K-Y}. 
This remark is important to compute the constant in Proposition \ref{prop.constant-at-infinity}. 
Any subgroup of $\calm(R)$ acts on $\frkJ(R)$ in a usual manner.
We note that there is an automorphism $g \longrightarrow g^*$ of $\calm(R)$ of order $2$ with the identity
\[(g\cdot X,g^*\cdot Y)=(g^*\cdot X,g\cdot Y)=(X,Y) \text{ for } X,Y \in \frkJ(R).\]
We have $g\cdot (X \times Y)=(g^*\cdot X) \times (g^*\cdot Y)$.
For any $\epsilon \in R^\times$ there is an element $g \in \calm(R)$ such that $\nu(g)=\epsilon$. As an example, the
$R$-linear transformation
\[\gamma(\epsilon):\frkJ(R) \ni \begin{pmatrix} a & x & y \\ \bar x & b & z \\ \bar y & \bar z & c\end{pmatrix} \longrightarrow 
\begin{pmatrix} \epsilon a & \epsilon x & y \\ \epsilon \bar  x& \epsilon b & z \\ \bar y & \bar z & \epsilon^{-1}c\end{pmatrix} \in \frkJ(R)\]
satisfies the required condition. Let $e_{ij}$ be the $3 \times 3$ matrix whose $(a,b)$-component is 1 for $(a,b)=(i,j)$ and $0$ otherwise. For $x \in \frkC(R)$, let $m_{xe_{ij}}$ be the $R$-linear transformation of $\frkJ(R)$ defined by
\[m_{xe_{ij}}\cdot X=(1_3+\bar xe_{ji}) X (1_3+xe_{ij}) \quad  \text{(usual matrix multiplication)}.\]
Then $m_{xe_{ij}}$ is an element of $\calm(R)$ such that $\nu(m_{xe_{ij}})=1$.
Put ${\bf M}=\calm \otimes_{\ZZ} \QQ$ and ${\bf M}'=\calm' \otimes_{\ZZ} \QQ$. Then ${\bf M}$ is an algebraic group over $\QQ$ of type $GE_{6,2}$ and ${\bf M}'$ is the derived group of ${\bf M}$, which is a simple group of type $E_{6,2}$.
Therefore we call $\calm$ and $\calm'$ the group schemes of type $GE_{6,2}$ and $E_{6,2}$, respectively.  

Recall the algebraic group $\bold G$ of type $E_{7,3}$ over $\QQ$ as in \cite{Bai}: Let $\bold X, \bold X'$ be two $\QQ$-vector spaces, each isomorphic to $\frkJ_{\QQ}$, and $\Xi, \Xi'$ be copies of $\QQ$. Let $\bold W=\bold X\oplus \Xi\oplus \bold X'\oplus \Xi'$, and for $w=(X,\xi,X',\xi')\in\bold W$,
define a quartic form $Q$ on $\bold W$ by 
$$Q(w)=(X\times X, X'\times X')-\xi \det(X)-\xi' \det(X')-\frac 14( (X,X')-\xi\xi')^2,
$$
and a skew-symmetric bilinear form $\{\,,\,\}$ by
$$\{w_1,w_2\}=(X_1,X_2')-(X_2,X_1')+\xi_1\xi_2'-\xi_2\xi_1'.
$$
Then 
$$\bold G(\QQ)=\left\{ g\in GL(\bold W_{\QQ}) |\, Qg=Q,\, g\{\,,\,\}=\{\,,\,\}\right\}.
$$
The center of $\bold G(\Bbb R)$ is $\{\pm \text{id}\}$ and the quotient of $\bold G(\Bbb R)$ by its center is the group of holomorphic automorphisms of $\frak T$. The real rank of $\bold G$ is 3, and it is split over $\Bbb Q_p$ for any prime $p$, and $\bold P=\bold M\bold N$ is the Siegel parabolic subgroup of $\bold G$. 

For $T \in \frak J(R)$, we define a group scheme  $\calu_T$ over $R$ by
\[\calu_T(S)=\{ g \in  \calm(S) \ | \ g\cdot T=T \}\]
for any commutative $R$-algebra $S$. 
By definition, $\calu_T(S)\subset \calm'(S)$. 
In particular, for $T \in \frkJ(\ZZ)_{>0}$, put ${\bf U}_T=\calu_T \otimes_{\ZZ} \QQ$. It is easy to see that 
${\bf U}_T$ is a connected regular algebraic group over $\QQ$ by (geometric) fiberwise argument. 
Further, ${\bf U}_T$ is
an exceptional group of type $F_4$ \cite[p.108]{Mars}. Therefore we call $\calu_T$ the group scheme of type $F_4$. In this section, we will prove the mass formula for ${\bf U}_T$ (Theorem \ref{th.mass-formula2}).

We introduce the symbol $\infty$ so that 
$a < \infty$ for any $a \in \ZZ$ and $p^{\infty}=0$. Then any $T \in \frkJ(\QQ_p)$ is $\calm(\ZZ_p)$-equivalent to
$\epsilon_1p^{a_1} \bot \epsilon_2p^{a_2} \bot \epsilon_3 p^{a_3}$ with $a_1,a_2,a_3 \in \ZZ \cup \{\infty \}, \ a_1 \le a_2 \le a_3$, and $\epsilon_i \in \ZZ_p^\times$. We put $e_i(T)=a_i$. We note that $e_i(T) \ (i=1,2,3)$ is uniquely determined by the $\calm(\ZZ_p)$-equivalence class of $T$, and that $a_i \ge 0$ if $T \in \frkJ(\ZZ_p)$. Similarly, any $T\in \frkJ(\ZZ_p/p^n\ZZ_p)$ is $\calm(\ZZ_p/p^n\ZZ_p)$-equivalent to
$\epsilon_1p^{a_1} \bot \epsilon_2p^{a_2} \bot \epsilon_3 p^{a_3}$ with $a_1,a_2,a_3 \in \{ 0,1,\ldots,n-1, \infty \}, \ a_1 \le a_2 \le a_3$, and $\epsilon_i \in (\ZZ_p/p^n\ZZ_p)^\times$. We again put $e_i(T)=a_i$. 
Again $e_i(T) \ (i=1,2,3)$ is uniquely determined by the $\calm(\ZZ_p/p^n\ZZ_p)$-equivalence class of $T$. 
For $r=(r_1,r_2,r_3) \in (\RR^{\times})^3$, we define an element $\theta(r)$ of $GL(\frkJ(\RR))$ by
\[\theta(r):\frkJ(\RR) \ni \begin{pmatrix} a & x & y \\ \bar x & b & z \\ \bar y & \bar z & c \end{pmatrix} \mapsto 
\begin{pmatrix} r_1^2a & r_1r_2x & r_1r_3 y \\  r_1r_2 \bar x & r_2^2b & r_2 r_3 z \\  r_1r_3 \bar y &  r_2 r_3 \bar z & r_3^2c \end{pmatrix} \in \frkJ(\RR).\]
Then $\theta(r)$ is an element of ${\bf M}(\RR)$ with $\nu(\theta(r))=(r_1r_2r_3)^2$. We denote by ${\bf M}^+(\RR)$ the subgroup of ${\bf M}(\RR)$ generated by 
${\bf M}'(\RR)$ and $\theta((\RR^{\times})^3)$.

\begin{lemma} \label{lem.Jordan-decomposition}
\begin{itemize}
\item[(1)] Let $K=\RR$ or $\QQ$, and $T \in \frkJ(K)_{>0}$. 
\begin{itemize}
\item[(1.1)] Then $T \sim_{{\bf M}(K)} 1_3.$
In particular, for $T \in R_3^+(\RR)$ we have $T \sim_{{\bf M}^+(\RR)} 1_3.$
\item[(1.2)] Suppose that $\det T=d$. Then $T \sim_{{\bf M}'(K)} 1 \bot 1 \bot d.$
\end{itemize}
\item[(2)] Let $T \in \frkJ(\QQ_p)^{\rm ns}$. 
\begin{itemize}
\item[(2.1)] Then we have  
 \[T \sim_{{\bf M}(\QQ_p)} 1_3,\quad T \sim_{\calm(\ZZ_p)} p^{e_1(T)} \bot p^{e_2(T)} \bot p^{e_3(T)}.\]
\item[(2.2)] Suppose that  $\det T=d$.  Then,
\[T \sim_{{\bf M}'(\QQ_p)} 1 \bot 1 \bot d,\quad T \sim_{\calm'(\ZZ_p)} p^{e_1(T)} \bot p^{e_2(T)} \bot \epsilon p^{e_3(T)} \]
with $\epsilon =d p^{-e_1(T)-e_2(T)-e_3(T)}$. Here $\epsilon$ is uniquely determined by $T$.
\end{itemize}
\item[(3)] Let $T \in \frkJ(\ZZ_p/p^n\ZZ_p)$. Then 
\[T \sim_{\calm(\ZZ_p/p^n\ZZ_p)} p^{e_1(T)} \bot p^{e_2(T)} \bot p^{e_3(T)}.\]
\end{itemize}
\end{lemma}
\begin{proof}
The assertion (1) and the first part of (2.1) follow from \cite[Proposition 1]{Mars}. 
To prove the second part of  (2), let $T \in \frkJ(\QQ_p)$ with $\det T =p^{e_1(T)+e_2(T)+e_3(T)}\epsilon$ and $\epsilon \in \ZZ_p^\times$. Then there is an element $g_1 \in \calm(\ZZ_p)$ with $\nu(g_1)=\epsilon$, and hence $\det (g_1\cdot T)=1$. The $(1,1)$-component of $g_1 \cdot T$ is expressed as $p^{a_1}\epsilon_1$ with $\epsilon_1 \in \ZZ_p^\times$, $a_1=e_1(T)$. Then there is an element $ \xi \in \frkC(\ZZ_p)$ such that $a_1+N(\xi)=1$. Then,
$m_{\xi e_{21}}\cdot T$ is of the form $p^{e_1(T)} \begin{pmatrix} 1 & x & y \\
\bar x & * &* \\
\bar y & * &*
\end{pmatrix}$ with $x,y \in \frkC(\ZZ_p)$. Then
$(m_{-xe_{12}}m_{-ye_{13}})\cdot (m_{\xi e_{21}}\cdot T)$ is of the form $ \begin{pmatrix} p^{e_1(T)} & 0 & 0 \\
\bar 0 & p^{e_2(T)} b & p^{e_2(T)} z\\
\bar 0 &  p^{e_2(T)}\bar z &p^{e_2(T)}c
\end{pmatrix}$ with $b,c \in \ZZ_p$, $z \in \frkC(\ZZ_p)$. Repeating this process, there is an element $g' \in \calm(\ZZ_p)$ such that $\nu(g')=1$ and 
$g'\cdot (g_1 \cdot T)=  \begin{pmatrix} p^{e_1(T)} & 0 & 0 \\
\bar 0 & p^{e_2(T)}  &0 \\
\bar 0 &0 & c'
\end{pmatrix}$ with $c' \in \QQ_p$. Since $\det (g_1\cdot T)=p^{e_1(T)+e_2(T)+e_3(T)}$, we have $c'=p^{e_3(T)}$.
This proves the second part of the assertion (2.1). The assertions (2.2) and (3) can be proved similarly. 
\end{proof}

\begin{corollary}
\label{cor.lifting-equivalence}
Let $T_1$ and $T_2$ be elements of $\frkJ(\ZZ_p)^{\rm ns}$, and let $n$ be an integer such that $n \ge e_3(T_1)+1$.
Suppose that $T_2 \equiv g\cdot T_1 \text{ mod } p^n \frkJ(\ZZ_p)$ with some $g \in \calm(\ZZ_p)$. Then, we have $T_2 \sim_{\calm(\ZZ_p)} T_1$. In particular, if $T_2 \equiv T_1 \text{ mod } p^n\frkJ(\ZZ_p)$, then $T_2 \sim_{\calm(\ZZ_p)} T_1$.
\end{corollary}

Let $T$ be an element of $\frkJ(\ZZ)_{>0}$. For $T' \sim_{\calm'(\ZZ)} T$, we say that $T'$ belongs to the same ${\bf M}_\AAA$-genus as $T$ and write
 $T' \approx T$ if $T' \sim_{\calm'(\ZZ_p)} T$ for any prime number $p$.
For $T \in {\mathfrak J}(\ZZ)_{>0}$, let 
\[\calg(T)=\{T'\in {\mathfrak J}(\ZZ)_{>0} \ | \ T' \approx T \}.\]  Put 
\begin{equation}\label{mass-def}
\mathrm{Mass}(T)=\sum_{T' \in \calg(T)/\calm'(\ZZ)} {1 \over \epsilon(T')},
\end{equation}
where $\epsilon(T')=\#\calu_{T'}(\ZZ)$.
For an algebraic variety $V$ over $\QQ$ of dimension $n$, let 
$\Omega_{V}$  the $\QQ$-vector space of  rational differential forms of degree $1$. We also define the top form  $\displaystyle \wedge^\mathrm{top} \Omega_{V}$ on $V$ as $\displaystyle \wedge^n \Omega_{V}$.
Let $H$ be an algebraic group over $\QQ$ acting on $V$ on the left.
Then we denote by $(\displaystyle \wedge^\mathrm{top} \Omega_{V})^H$ the $\QQ$-vector space of $H$-invariant rational differential forms of degree $n$.  In particular, if $V=H$, we write
$(\displaystyle \wedge^\mathrm{top} \Omega_{H})^H$ as $(\displaystyle \wedge^\mathrm{top} \Omega_{H})^\mathrm{inv}$. 
Let $\calp$ be the set of all prime numbers. For the symbol $\infty$ we make the convention that $\ZZ_{\infty}=\QQ_{\infty}=\RR$. 
From now on, we normalize the measure $|dt|_p$ on $\QQ_p$ for $p \in \calp$ so that
\[\int_{\ZZ_p} |dt|_p=1.\]
We also take the usual Euclidean measure $dt$ on $\RR$ as $|dt|_{\infty}$.

For an element $\omega \in \displaystyle \wedge^\mathrm{top} \Omega_{V}$ and $p \in \calp \cup \{\infty\}$, we denote by $|\omega|_p$ the measure 
$V \otimes_{\QQ} \QQ_p$  induced from $\omega$. 
For an element $\omega \in (\displaystyle \wedge^\mathrm{top} \Omega_{V})^H$ define the Tamagawa measure $|\omega_\AAA|$ on $V_{\AAA}$ as
\[|\omega|_\AAA=|\omega|_{\infty}\prod_{p < \infty} \lambda_p|\omega|_p,\]
where $\{\lambda_p\}_{p \in \calp}$ is a set of convergence factors.
We note that $|\omega|_{\AAA}$ does not depend on the choice of $\omega$ but depends on $\{\lambda_p\}_{p \in \calp}$.
Let $H$ be an algebraic group, and $\omega$ the Tamagawa measure on $H$. We then define the Tamagawa number $\tau(H)$ as 
\[\tau(H)=\int_{H_\AAA/H(\QQ)} |\omega|_\AAA.\]

Now in order to give the mass formula for the exceptional group of type $F_4$, for an element $T \in \frkJ(\ZZ)_{>0}$, we take an element  $\omega_T$  of $(\displaystyle \wedge^\mathrm{top} \Omega_{{\bf U}_T})^\mathrm{inv}$ suitably.
Let $H$ be a group (resp. a group scheme) acting on a set (resp. a scheme) $S$. Then, for $s \in S$, we denote by $\calo_H(s)$ the orbit (resp. the orbit scheme) of $s$ under $H$, that is
\[\calo_H(s)=\{g\cdot s \ | \ g \in H \}.\]

We take $\omega_T$ on ${\bf U}_T$ in the following way:
We take non-zero elements $dg \in (\wedge^\mathrm{top} \Omega_{\bf M})^\mathrm{inv}$ and $\eta_T=\eta_T(x) \in (\wedge^\mathrm{top} \Omega_{{\bf M}/{\bf U}_T})^{\bf M}$, and define an element $\omega_T \in (\wedge^\mathrm{top} \Omega_{{\bf U}_T})^\mathrm{inv}$ as
\[\omega_T=dg/\eta_T.\]
To be more precise, we observe the natural quotient $\pi:{\bf M}\longrightarrow {\bf M}/{\bf U}_T$ which is smooth by 
fiberwise argument. Note that the quotient ${\bf M}/{\bf U}_T$ does exist by \cite[Th\'eor\`eme 10.1.2]{SGA3-I} 
and it is also regular by fiberwise argument again. 
Applying \cite[Proposition 5 in Chapter 2, p.36]{BLR}, we have the following exact  sequence of locally free sheaves on ${\bf M}$:
$$0\longrightarrow \pi^\ast \varOmega_{{\bf M}/{\bf U}_T}\longrightarrow \varOmega_{{\bf M}}\longrightarrow 
\varOmega_{{\bf M}/({\bf M}/{\bf U}_T)}\longrightarrow 0,$$
where $\varOmega_{{\bf M}/{\bf U}_T}$ and $\varOmega_{{\bf M}}$ denote the sheaves of differentials on ${\bf M}/{\bf U}_T$ and ${\bf M}$, respectively, and $\varOmega_{{\bf M}/({\bf M}/{\bf U}_T)}$ denote the sheaf of differentials on ${\bf M}$ over ${\bf M}/{\bf U}_T$. 
Then it follows from fiberwise argument that $\varOmega_{{\bf M}/({\bf M}/{\bf U}_T)}=i_\ast \varOmega_{{\bf U}_T}$ 
where $i:{\bf U}_T\hookrightarrow {\bf M}$ is the natural inclusion. 
The above exact sequence yields, on top forms, 
$$\wedge^\mathrm{top} \varOmega_{\bf M}=\wedge^\mathrm{top}\pi^\ast\varOmega_{{\bf M}/{\bf U}_T}
\otimes_{\mathcal{O}_{\bf M}}\wedge^\mathrm{top}i_\ast \varOmega_{{\bf U}_T}.$$
Since $\wedge^\mathrm{top} \varOmega_{\bf M}$, $\wedge^\mathrm{top}\pi^\ast\varOmega_{{\bf M}/{\bf U}_T}$, and 
$\wedge^\mathrm{top}i_\ast \varOmega_{{\bf U}_T}$ are all invertible, 
we have the desired form on ${\bf U}_T$ as above. 
 
Now consider the following exact sequence
\[1 \longrightarrow {\bf M}' \longrightarrow {\bf M} \longrightarrow GL_1 \longrightarrow 1.\]
Let ${dt \over t}$ be the canonical invariant differential form on $GL_1$. Then we can define the differential form $dg'$ on ${\bf M}'$ by ${dt \over t}=dg/dg'$. We can also define the differential form $dg_T'$ on ${\bf M}'/{\bf U}_T$ in a similar way.
 We note that we also have
\[\omega_T=dg'/\eta_T'.\]
 Let $\calm_p=\calm \otimes_{\ZZ} \ZZ_p$ and $\calm'_p=\calm' \otimes_{\ZZ} \ZZ_p$.

\begin{lemma} \label{lem.order-of-GE_6}
Let $n$ be a positive integer. Then,
\[\#\calm(\ZZ_p/p^n\ZZ_p)=p^{79(n-1)} \#\calm(\ZZ_p/p\ZZ_p),\]
and 
\[\#\calm'(\ZZ_p/p^n\ZZ_p)=p^{78(n-1)} \#\calm(\ZZ_p/p\ZZ_p).\]
Moreover we have 
\[\#\calm(\ZZ_p/p\ZZ_p)=p^{36}(p^{12}-1)(p^9-1)(p^8-1)(p^6-1)(p^5-1)(p^2-1)(p-1).\]
and
\[\#\calm'(\ZZ_p/p\ZZ_p)=p^{36}(p^{12}-1)(p^9-1)(p^8-1)(p^6-1)(p^5-1)(p^2-1).\] 

\end{lemma}
\begin{proof}
The schemes $\calm_p$ and $\calm'_p$ are  smooth over $\ZZ_p$ and their  dimensions are $79$ and $78$, respectively. Thus the first assertion can be proved by a standard method. Since ${\bf M}$ and ${\bf M}'$ are algebraic groups over $\QQ$ of types  $GE_6$ and $E_{6,2}$, respectively, we have
\[\#\calm(\ZZ_p/p\ZZ_p)=(p-1)\#\calm'(\ZZ_p/p\ZZ_p) \]
and
\[\#\calm'(\ZZ_p/p\ZZ_p)=\mathrm{gcd}(3,p-1) \#E_6(p),\]
where $E_6(p)$ is a finite Chevalley group of type $E_6$ over $\FF_p$. Thus the assertion follows from
\cite[Theorem 9.4.10]{Ca}.
\end{proof}
 
From now on, for each prime number $p$, put 
$\delta_p= (1-p^{-2})(1-p^{-5})(1-p^{-6})(1-p^{-8})(1-p^{-9})(1-p^{-12})$. 
\begin{lemma} \label{lem.normalization-of-dg}
Let $dg$ and $dg'$ be as above. Then for any prime number $p$, we have
\[\int_{{\calm}(\ZZ_p)} |dg|_p=(1-p^{-1})\int_{{\calm}'(\ZZ_p)} |dg'|_p. \]
Moreover,  there exists a finite set $\mathcal S$ of prime numbers (depending on $dg$) such that  for any prime number $p \not\in \mathcal S$, we have
\[\int_{{\calm}'(\ZZ_p)} |dg'|_p=\delta_p. \]
\end{lemma}
\begin{proof} The first assertion follows from the definition of $dg$ and $dg'$. The second assertion follows from Lemma \ref{lem.order-of-GE_6} and \cite[Theorem 2.2.5]{We}.

\end{proof}
\begin{lemma} \label{lem.volume-of-E_{62}}
Let
\[v(\calm'(\ZZ))=\int_{{\bf M}'(\RR)/\calm'(\ZZ)} |dg'|_\infty,\]
and $\displaystyle d_0=\prod_{p \in \mathcal S} {\delta_p \over  \int_{{\calm}'(\ZZ_p)} |dg'|_p}$.
Then we have 
\[v(\calm'(\ZZ))=d_0\zeta(2)\zeta(5)\zeta(6)\zeta(8)\zeta(9)\zeta(12).\]
\end{lemma}
\begin{proof}
Since ${\bf M}'$ is simply connected and ${\bf  M}'(\RR)$ is not compact, the strong approximation holds (cf. \cite{P}), and so we have
\[{\bf M}'_\AAA={\bf M}'(\RR)\prod_{p<\infty} \calm'(\ZZ_p){\bf M}'(\QQ).\]
Moreover, by \cite{Mars} we have $\tau({\bf M}')=1$. Hence by Lemma \ref{lem.normalization-of-dg}, we have
\begin{align*}
\tau({\bf M}')&=\int_{{\bf M}'_\AAA/{\bf M}'(\QQ)} |dg'|_\AAA =\prod_{p <\infty} \int_{\calm'(\ZZ_p)} |dg'|_p \int_{{\bf M}'(\RR)/\calm'(\ZZ)} |dg'|_\infty\\
&=\prod_{p <\infty} (1-p^{-2})(1-p^{-5})(1-p^{-6})(1-p^{-8})(1-p^{-9})(1-p^{-12})\times d_0^{-1} v(\calm'(\ZZ))=1.
\end{align*}
This proves the assertion.
\end{proof}

We define a scheme $\calj$ over $\ZZ$ and its subscheme $ \calj^{\rm ns}$ by 
\[\calj(R)=\frkJ(R) \text{ and } \calj^{\rm ns}(R)=\frkJ(R)^{\rm ns}\]
for a commutative algebra $R$ and put ${\bf J}=\calj \otimes_{\ZZ} \QQ$ and ${\bf J}^{\rm ns}=\calj^{\rm ns} \otimes_{\ZZ} \QQ$, respectively. We use the same symbol $\omega$ to denote the restriction of $\omega \in \wedge^\mathrm{top} \Omega_{\bf J}$ to ${\bf J}^{\rm ns}$.
We regard ${\bf J}$ as the $27$-dimensional affine space with coordinates $x_1,\ldots,x_{27}$, and we define a differential form
$d\sigma(x)$ and an  ${\bf M}$-invariant differential form $\eta=\eta(x)$ on ${\bf J}$ as $d\sigma(x)=\wedge_{i=1}^{27} dx_i$, and $\eta(x)=(\det x)^{-9} d\sigma(x)$, respectively. 
We note that 
\[\int_{{\bf J}(\ZZ_p)} |d\sigma(x)|_p=1.\]
 Now take a non-zero element $T$ of $\calj(\ZZ)_{>0}$, and let 
${\bf J}_T$  denote the orbit of $T$ under ${\bf M}$.
We denote by $\eta(g\cdot T) \ (g \in {\bf M})$ the restriction of $\eta$ to ${\bf J}_T$.
Let 
\[f_T: {\bf M} \in g \longrightarrow g\cdot T \in {\bf J}^{\rm ns}\]
be a morphism of varieties. 
Then, $f_T$ induces an  ${\bf M}$-equivariant  isomorphism from ${\bf M}/{\bf U}_T$ to ${\bf J}_T$ 
(cf. \cite[Th\'eor\`eme 10.1.2]{SGA3-I} for the existence of the quotient ${\bf M}/{\bf U}_T$ as a variety over $\QQ$ and 
for $f_T$ to be revealed as an isomorphism).
 With this isomorphism, we identify ${\bf M}/{\bf U}_T$ with 
${\bf J}_T$ and  we take  $\eta(g\cdot T) $ as an element $\eta_T$ of $(\wedge^{\rm top}\Omega_{{\bf M}/{\bf U}_T})^{\bf M}$. Then we can choose the differential form $dg$ on ${\bf M}$ 
so that
\begin{equation}\label{I}
\int_{{\bf M}(\QQ_v)}f(g)|dg|_v=\int_{\calo_{{\bf M}}(T)} |\eta(g\cdot T)|_v\int_{{\bf U}_T}f(gh)|\omega_T(h)|_v,
\end{equation}
for any (finite or infinite) place of $\QQ$ and $f \in \mathrm{L}^1({\bf M}(\QQ_v),|dg|_v)$ (cf. \cite[page 145, line 17]{Sa}). Here we note that ${\bf M}$ is unimodular.
We also denote by ${\bf J}_T'$  by  the orbit of $T$ under ${\bf M}'$.
Let 
\[f_T': {\bf M}' \in g \longrightarrow g\cdot T \in {\bf J}^{\rm ns}\]
be a morphism of varieties. 
Then, $f_T'$ induces an  ${\bf M}'$-equivariant isomorphism from ${\bf M}'/{\bf U}_T$ to ${\bf J}_T'$. Let $dg'$ be the differential form of ${\bf M}'$ constructed from $dg$ as before. Then, in the same way as above, we can choose a non-zero element  $\omega'$  of  $(\wedge^{\rm top}\Omega_{{\bf M}'/{\bf U}_T})^{{\bf M}'}$ so that
\begin{equation}\label{I'}
\int_{{\bf M}'(\QQ_v)}f(g')|dg'|_v=\int_{\calo_{{\bf M}'}(T)} |\omega'(g'\cdot T)|_v\int_{{\bf U}_T}f(g'h)|\omega_T(h)|_v,
\end{equation}
for any (finite or infinite) place of $\QQ$ and $f \in \mathrm{L}^1({\bf M}'(\QQ_v),|dg'|_v)$.

For $T \in \calm(\ZZ_p)^{\rm ns}$, we note that $\int_{\calo_{\calm(\ZZ_p)}(T)} |d\sigma(x)|_p \not=0$.
\begin{definition}
\label{def.local-density}  For $T \in \calm(\ZZ_p)^{\rm ns}$ 
we define the local density  $\beta_p(T)$ of $T$ by
\[\beta_p(T)={ (1-p^{-1})\delta_p \over  \int_{\calo_{\calm(\ZZ_p)}(T)} |d\sigma(x)|_p }.\]
We also define $\alpha_p(T)$ by
\[\alpha_p(T)=\int_{\calo_{\calm'(\ZZ_p)}(T)} |\omega'(g' \cdot T)|_p.\] 
\end{definition}
By definition, $\alpha_p(T)$ is uniquely determined by the $\calm'(\ZZ_p)$-equivalence class of $T$, and $\beta_p(T)$ is uniquely determined by the $\calm(\ZZ_p)$-equivalence class of $T$.
This is an analogue of the local density of a quadratic form as 
will be explained in Section 6.
We also note that $\#\calo_{\calm(\ZZ_p/p^n\ZZ_p)}(T) \not=0$ for $T \in \frkJ(\ZZ_p/p^n\ZZ_p)^{\rm ns}$. 
For $T \in \frkJ(\RR)_{>0}$, we define $\beta_{\infty}(T)$ by
\[\beta_{\infty}(T)=\int_{{\bf U}_T(\RR)}|\omega_T|_{\infty}.\]

\begin{theorem} \label{th.interpretation-of-local-density}
\begin{itemize}
\item[(1)]  Let $p$ be a prime number and let  $T \in \frkJ(\ZZ_p)^{\rm ns}$. Then we have
\[{\alpha_p(T)}={|\det T|_p^{-9} \delta_p \over \beta_p(T)},\]
and for a positive integer $n$ such that $n \ge e_3(T)+1$, 
\begin{align*}
&\beta_p(T) = \delta_p(1-p^{-1}){ p^{27n} \over  \#(\calo_{\calm(\ZZ_p/p^n\ZZ_p)}(T))},
\end{align*}
where $\bar T=T \text{ mod } p^n\frkJ(\ZZ_p)$.
In particular, $\alpha_p(T)$ is uniquely determined by the $\calm(\ZZ_p)$-equivalence class of $T$.
\item[(2)] Let $T \in \frkJ(\RR)_{>0}$. Then $\beta_{\infty}(T)=c_0^{-1},$ where $c_0\ne 0$ is independent of $T$.
\end{itemize}
\end{theorem}
\begin{proof}
(1)  By applying the formula (\ref{I'}) to the case when $f$ is the characteristic function of $\ZZ_p$, we have
\begin{align*}
\int_{\calm'(\ZZ_p)}|dg'|_p=\int_{\calo_{\calm'}(T)} |\omega'(g'\cdot T)|_p\int_{\calu_T}|\omega_T|_p. 
\end{align*}
Similarly we have
\begin{align*}
\int_{\calm(\ZZ_p)}|dg|_p=\int_{\calo_{\calm}(T)} |\eta(g\cdot T)|_p\int_{\calu_T}|\omega_T|_p. 
\end{align*}
By Lemma \ref{lem.normalization-of-dg}, we have
\[\int_{\calo_{\calm}(T)} |\eta(g\cdot T)|_p=(1-p^{-1})\int_{\calo_{\calm'}(T)} |\omega'(g'\cdot T)|_p.\]
We note that $|\det x|_p=|\det T|_p$ for any $x \in  \calo_{\calm(\ZZ_p)}(T)$, and hence we have
\begin{align*}
\int_{\calo_{\calm(\ZZ_p)}(T)} |\eta(g \cdot T)|_p&=\int_{\calo_{\calm(\ZZ_p)}(T)} |\det x|_p^{-9}|d\sigma(x)|_p =|\det T|_p^{-9}\int_{\calo_{\calm(\ZZ_p)}(T)} |d\sigma(x)|_p. 
\end{align*}
This proves the first assertion. 

Now for a positive integer $n$, we have
\[\int_{\calo_{\calm(\ZZ_p)}(T)} |d\sigma(x)|_p=\sum_{\bar a \in \frkJ(\ZZ_p/p^n\ZZ_p)}  \int_{(a+p^n\frkJ(\ZZ_p)) \cap \calo_{\calm(\ZZ_p)}(T)} |d\sigma(x)|_p.\]
By Corollary \ref{cor.lifting-equivalence}, we have
$(a+p^n\frkJ(\ZZ_p)) \cap \calo_{\calm(\ZZ_p)}(T)=a+p^n\frkJ(\ZZ_p)$ or $\emptyset$, and 
$\int_{a+p^n\frkJ(\ZZ_p)} |d\sigma(x)|_p=p^{-27n}$. 
Hence we have
\[\int_{\calo_{\calm(\ZZ_p)}(T)} |d\sigma(x)|_p=p^{-27n}\sum_{\bar a \in \frkJ(\ZZ_p/p^n\ZZ_p) \atop \bar a \sim_{\calm(\ZZ_p/p^n\ZZ_p)} \bar T} 1=p^{-27n}\#(\calo_{\calm(\ZZ_p/p^n\ZZ_p)}(\bar T)).\]

(2) We take $\phi_T(g)=\exp(-{\rm tr}(g \cdot T)) \chi_{{\bf M}^+(\RR)}(g)\ (g \in {\bf M}(\RR))$, where $\chi_{{\bf M}^+(\RR)}(g)$ is the characteristic function of ${\bf M}^+(\RR)$.Then, 
by (\ref{I}), we have
\[\int_{{\bf M}(\RR)} \phi_T(g)|dg|_{\infty}= \int_{{\bf M}(\RR)\cdot T}|\eta(g \cdot T)|_\infty 
\int_{{\bf U}_T(\RR)} \phi_T(gu)|\omega_T|_\infty.\]
Since we have $\phi_T(gu)=\phi_T(g)$ for any $g \in {\bf M}(\RR)$ and $u \in {\bf U}_T(\RR)$, $\phi_T(g)$ is uniquely determined by the orbit $g\cdot T$, and we write it as $\bar \phi(g\cdot T)$. Then
we have
\[\int_{{\bf M}(\RR)} \phi_T(g)|dg|_{\infty}= \int_{{\bf M}(\RR)\cdot T}|\bar \phi( g\cdot T)\eta(g \cdot T)|_\infty 
\int_{{\bf U}_T(\RR)} |\omega_T|_\infty.\]
That is, we have 
\[\beta_{\infty}(T)= { \int_{{\bf M}(\RR)} \phi_T(g)|dg|_{\infty} \over \int_{{\bf M}(\RR) \cdot T} \bar \phi(g\cdot T)|\eta(g \cdot T)|_\infty}.\]
There is an element $g_0 \in {\bf M}^+(\RR)$ such that
$T=g_0 \cdot 1_3$. Then, $\phi_T(g)=\phi_{1_3}(gg_0)$. Since
$|dg|_{\infty}$ is  also right invariant, we have 
\[\int_{{\bf M}(\RR)} \phi_T(g)|dg|_{\infty} =\int_{{\bf M}(\RR)} \phi_{1_3}(g)|dg|_{\infty} .\]
Moreover,  we have the measure preserving mapping 
\[{\bf M}(\RR)\cdot T \ni  g \cdot T \longrightarrow  {g_0^{-1}gg_0} \cdot 1_3 \in {\bf M}(\RR)\cdot 1_3,\]
and  $|\eta(g \cdot 1_3)|_\infty $ is right ${\bf M}(\RR)$-invariant.
 Hence we have 
\begin{align*}
\int_{{\bf M}(\RR) \cdot T}\bar \phi(g \cdot T)|\eta(g \cdot T)|_\infty &=\int_{{\bf M}(\RR)(g_0\cdot 1_3)}\bar \phi( {g g_0} \cdot 1_3)|\eta(g\cdot (g_0 \cdot1_3)|_{\infty}\\
&=\int_{(g_0^{-1}{\bf M}(\RR)g_0) 1_3}\bar \phi({g_0^{-1} gg_0}\cdot 1_3)|\eta((g_0^{-1}gg_0) \cdot1_3)|_{\infty} \\
&=\int_{{\bf M}(\RR)\cdot 1_3}\bar \phi(g \cdot 1_3)|\eta(g \cdot 1_3)|_\infty
\end{align*}
Hence, putting $c_0=\beta_\infty(1_3)^{-1}$, we have
\[\beta_{\infty}(T)= {\int_{{\bf M}(\RR)} \phi_{1_3}(g)|dg|_{\infty} \over \int_{{\bf M}(\RR) \cdot 1_3} \bar \phi(g \cdot 1_3)|\eta(g \cdot 1_3)|_\infty}=c_0^{-1}.\]
\end{proof}
\begin{theorem}\label{th.mass-formula2}(Mass-formula)
Let $T$ be an element of $\frkJ(\ZZ)_{>0}$. Put $c=c_0d_0$, where $d_0$ is as in Lemma \ref{lem.volume-of-E_{62}}. Then we have
\[\mathrm{Mass}(T)=c{(\det T)^9 \over \prod_{p<\infty} \beta_p(T)}.\]
\end{theorem}
\begin{proof} 
Let $T_1,\ldots,T_h$ be a complete set of $\calm'(\ZZ)$-classes in the genus of $T$. 
Since ${\bf U}_T$ is an exceptional group of type $F_4$, which is connected and semi-simple, by \cite[Corollary 2.2]{Sa}, we have
\begin{align*}
\tau({\bf U}_T) \prod_{p <\infty}\alpha_p(T) &={1 \over v(\calm'(\ZZ))} \sum_{i=1}^h  \int_{{\bf U}_{T_i}(\RR)/\calu_{T_i}(\ZZ)} |\omega_{T_i}|_\infty ={1 \over v(\calm'(\ZZ))} \sum_{i=1}^h {\beta_{\infty}(T_i) \over \# \calu_{T_i}(\ZZ)}.
\end{align*}
By \cite{Mars}, we have $\tau({\bf U}_T)=1$.
Thus the assertion follows from Lemma \ref{lem.volume-of-E_{62}} and Theorem \ref{th.interpretation-of-local-density}.
\end{proof}
We will determine the constant $c$ in Proposition \ref{prop.constant-at-infinity}. 
For $p \le \infty$, let $\iota_p:\frkJ(\QQ) \longrightarrow \frkJ(\QQ_p)$ be the natural embedding, and 
let $\varphi: \frkJ(\QQ) \longrightarrow   \prod_{p \le \infty} \frkJ(\QQ_p)$ be the diagonal embedding.
\begin{proposition} \label{prop.approximation}
\begin{itemize}
\item[(1)] Let $(T_p)$ be an element of $\prod_{p \le \infty} \frkJ(\QQ_p)$. Suppose that there is a positive integer $d$ such that $\det T_p=d$  for any $p$. Then
there is an element $T \in \frkJ(\QQ)_{>0}$ such that
$\iota_p(T)=g_p\cdot T_p$ with some $g_p \in \bfM(\QQ_p)$ for any $p$. 
\item[(2)] Let $T \in \frkJ(\ZZ_p)^{\rm ns}$. Then there is an element $\widetilde T \in \frkJ(\ZZ)_{>0}$ such that
\[\widetilde T \sim_{\calm(\ZZ_p)} T \text{ and } \ord_q(\det \widetilde T)=0 \text{ for } q \not=p.\]
\end{itemize}
\end{proposition}
\begin{proof}
(1) By Lemma \ref{lem.Jordan-decomposition} (1), we have 
\[T_p \sim_{\bf M'(\QQ_p)} 1 \bot 1 \bot d.\]
This proves the assertion.

(2) Put $\widetilde T=p^{e_1(T)} \bot p^{e_2(T)} \bot p^{e_3(T)}$. Then, by Lemma \ref{lem.Jordan-decomposition} (2), $\widetilde T$ satisfies the required condition.
\end{proof}

Let
\[\JJ=\prod_p (\frkJ(\ZZ_p)/\calm'(\ZZ_p)).\]
Then $\varphi$  induces a mapping from $\frkJ(\ZZ)_{>0}/\prod_p \calm'(\ZZ_p)$ to $\JJ$, which will be denoted also by $\varphi$. For  $d \in \ZZ_p \setminus \{0\}$, put
\[\frkJ(d,\ZZ_p)=\{ T \in  \frkJ(\ZZ_p) \ | \ \det T=d\}.\]
Moreover, for a positive integer $d$, put
\[\frkJ(d, \ZZ)=\{ T \in \frkJ(\ZZ) \ | \ \det T=d\},\]
and
\[\JJ(d)=\prod_p (\frkJ(d,\ZZ_p)/\calm'(\ZZ_p)).\]
\begin{proposition} \label{prop.local-global-exceptional-Jordan} The mapping $\varphi$ induces a bijection from
$\frkJ(d,\ZZ)_{>0}/\prod_p \calm'(\ZZ_p)$ to $\JJ(d).$
\end{proposition}
\begin{proof} It is clear that $\varphi$ is injective.  Let $(x_p) \in \JJ(d)$. Then, by Proposition \ref{prop.approximation}, there is an element 
$y \in \frkJ(\QQ)_{>0}$ such that $x_p=g_p\cdot y$ with some $g_p \in {\bf M}'(\QQ_p)$ for any prime number $p$.
For $p$ not dividing $d$, we may assume that $g_p \in \calm'(\ZZ_p)$. Hence $(g_p)$ defines an element of the finite part ${\bf M}'_{\AAA_f}$ of
${\bf M}_{\AAA}$. By the strong approximation theorem, 
\[{\bf M}'_{\AAA}= {\bf M}'(\RR)\prod_{p<\infty} \calm'(\ZZ_p){\bf M}'(\QQ).\]
Hence there exist elements $\gamma \in {\bf M}'(\QQ) , \gamma_{\infty} \in {\bf M}'(\RR)$ and $(\gamma_p) \in \prod_{p<\infty} 
\calm'(\ZZ_p)$ such that $(g_p)= \gamma_{\infty} (\gamma_p)\gamma$. Put $x=\gamma \cdot y$. Then $x$ belongs to $\frkJ(d,\ZZ)$ and
$\varphi(x)=(x_p)$. This proves the surjectivity of $\varphi$.
\end{proof}

\section{Analytic properties of Rankin-Selberg series}
For a complex number $x$, we put ${\bf e}(x)=\exp(2\pi \sqrt{-1}x)$. 
As in Section 2, ${\bf M}(\RR)$ acts on $\frkJ(\RR)$  in a usual manner, and $d^*Y=\det(Y)^{-9}dY$ is the ${\bf M}(\RR)$-invariant measure in $R_3^+(\Bbb R)$ with this action.
However, $g \cdot Y \not\in R_3^+(\RR)$ for $Y \in R_3^+(\RR)$ in general. Therefore, we define a new  action `$*$' of 
${\bf M}(\RR)$ on $\frkJ(\RR)$ as
\[{\bf M}(\RR) \times \frkJ(\RR)) \ni (g,Y) \mapsto g*Y=\nu(g)g\cdot Y \in \frkJ(\RR).\]
This action induces an action of ${\bf M}(\RR)$ on $R_3^+(\RR)$. For a subgroup $H$ of ${\bf M}(\RR)$ and a subset $\cals$ of  $R_3^+(\Bbb R)$, we denote by
$\cals//H$ the $H$-equivalence classes of $\cals$ under the action `$*$'. We also define an action of ${\bf M}(\RR)$ on $\frkT$ as 
\[\frkT \ni Z=X+\sqrt{-1}Y \mapsto g*Z=g*X+\sqrt{-1}g*Y.\]
Recall
$d^*Z=\det(Y)^{-18}dXdY$ is the ${\bf G}(\RR)$-invariant measure in $\frak T$, and $d^*Y=\det(Y)^{-9}dY$ is the ${\bf M}(\RR)$-invariant measure in $R_3^+(\Bbb R)$ with the action `$*$'.
Let $\frak F$ be a fundamental domain for the action of $\Gamma=\bold G(\Bbb Z)$ on $\frak T$.
Let $\frak R$ be a fundamental domain for the action of $\calm'(\ZZ)$ on $R_3^+(\Bbb R)$.
Let $\Gamma_\infty=\Gamma\cap \bold P(\Bbb Q)$.
Then 
$$\frak F_\infty=\{ Z=X+\sqrt{-1} Y\in \frak T \, | \, \text{$X$ mod 1}, Y\in \frak R\},
$$ 
is a fundamental domain for the action of $\Gamma_\infty$ on $\frak T$.

Let $F$ be a cusp form of weight $k$ on the exceptional domain $\frak T$ with respect to $\Gamma$, namely, $F$ is a holomorphic function on $\frak T$, and for $\gamma\in \Gamma$ and $Z\in \frak T$,
$$F(\gamma Z)=j(\gamma,Z)^k F(Z),
$$
where $j(\gamma,Z)$ is the canonical factor of automorphy, which satisfies the usual property: ${\rm det}(Im(\gamma Z))={\rm det}(Im(Z))|j(\gamma,Z)|^{-2}$.
Also we have the Fourier expansion
$$F(Z)=\sum_{T\in \frak J(\Bbb Z)_{>0}} a_F(T){\bf e}((T,Z)).
$$

Recall the Petersson inner product: For $F,G$ modular forms of weight $k$, one of them being a cusp form,
let
$$\langle F,G\rangle=\int_{\frak F} F(Z)\overline{G(Z)} \det(Y)^k\, d^*Z.
$$
%Let $\frkJ(\ZZ)_{>0}/\calm(\ZZ)$ be the set of equivalence classes under the action of $\calm(\ZZ)$.%
For two cusp forms $F,G$ of weight $k$, define 
$$R(s,F,G)=\sum_{T\in {\frak J}(\Bbb Z)_{>0}//\calm(\ZZ)} \frac {a_F(T)\overline{a_G(T)}}{\epsilon(T) \det(T)^s},
$$
where $\epsilon(T)=\#\calu_{T}(\ZZ)$ is as in (\ref{mass-def}).
We note that
$$R(s,F,G)=\sum_{T\in {\frak J}(\Bbb Z)_{>0}/\calm'(\ZZ)} \frac {a_F(T)\overline{a_G(T)}}{\epsilon(T) \det(T)^s}.
$$
Recall the Eisenstein series from \cite{Kim1}.
$$E(Z,s)=\det(Y)^s\sum_{\gamma\in \Gamma_\infty\backslash \Gamma} |j(\gamma,Z)|^{-2s}.
$$

\begin{theorem}\cite{Kim1}\label{th.Eisenstein-series} 
Let $\Psi(s)=\xi(2s)\xi(2s-4)\xi(2s-8)(2s-2)(2s-4)E(Z,s)$, where $\xi(s)=\pi^{-\frac s2}\Gamma(\frac s2)\zeta(s)$. Then $\Psi(s)$ can be continued to a meromorphic function in $s\in \Bbb C$ with a simple pole at 
$s=0,\frac 12,\frac 52,4,5,\frac {13}2,\frac {17}2,9$, and satisfies the functional equation $\Psi(9-s)=\Psi(s)$.
Only the residues at $s=0,9$ are constants. The residue of $E(Z,s)$ at $s=9$ is $\dfrac {\xi(5)\xi(9)}{28 \xi(10)\xi(14)\xi(18)}$.
\end{theorem}

We prove
\begin{theorem} \label{th.Rankin-Selberg} 
Suppose $F,G$ are cusp forms of weight $k$. Then for $Re(s)\gg 0$, $R(s,F,G)$ converges absolutely, and in this region we have the integral representation
$$\gamma(s) R(s,F,G)=\int_{\frak F} F(Z)\overline{G(Z)} E(Z,s+9-k) \det(Y)^{k}\, d^*Z,
$$
where $\gamma(s)=2^{-6s}\pi^{12-3s}\prod_{n=0}^2 \Gamma(s-4n)$. 
The analytic continuation and functional equation of $E(Z,s)$ give rise to those of $R(s,F,G)$:
Let 
$$\Lambda(s,F,G)=\gamma(s)R(s,F,G)\xi(2s+18-2k)\xi(2s+14-2k)\xi(2s+10-2k)(2s+16-2k)(2s+14-2k).
$$ 
Then
$$\Lambda(2k-9-s,F,G)=\Lambda(s,F,G).
$$
Furthermore, $R(s,F,F
)$ has a simple pole at $s=k$ with the residue
$$\langle F,F\rangle \frac {2^{6k-2}\pi^{3k-12}\prod_{i=1}^3 \Gamma(k-4i+4)^{-1}\xi(5)\xi(9)}{\xi(10)\xi(14)\xi(18)}.
$$
\end{theorem}
\begin{proof} By Hecke bound, $a_F(T)\ll \det(T)^{\frac k2}$. Hence for $\sigma=Re(s)$,
$$R(s,F,G)\ll \sum_{T\in \frkJ(\ZZ)_{>0}//\calm(\ZZ)} \det(T)^{-\sigma+k}.
$$
It converges absolutely for $\sigma\gg 0$.
Consider
\begin{equation}\label{main}
\Phi(s)=\int_{\frak F_\infty} F(Z)\overline{G(Z)} \det(Y)^{s+9}\, d^*Z.
\end{equation}
Then
$$\Phi(s)=\int_{\frak R} \det(Y)^s \left(\int_{\text{$X$ mod 1}} F(Z)\overline{G(Z)}\, dX\right)\, d^*Y.
$$

The inner integral is 
\begin{eqnarray*}
&& \int_{\text{$X$ mod 1}} \left(\sum_{T,T'\in \frak J(\Bbb Z)_{>0}} a_F(T)\overline{a_G(T)}e^{2\pi \sqrt{-1}(T-T',X)} e^{-2\pi(T+T',Y))}\right)\, dY \\
&& =\sum_{T\in \frak J(\Bbb Z)_{>0}} a_F(T)\overline{a_G(T)} e^{-4\pi (T,Y)}.
\end{eqnarray*}
Therefore,
\begin{eqnarray*}
&& \Phi(s)=\sum_{T\in \frak J(\Bbb Z)_{>0}} a_F(T)\overline{a_G(T)} \int_{\frak R} \det(Y)^s e^{-4\pi (T,Y)}\, d^*Y \\
&& \phantom{xxxx} =\sum_{T\in \frkJ(\ZZ)_{>0}//\calm(\ZZ)} \epsilon(T)^{-1} a_F(T)\overline{a_G(T)}\sum_{m\in \bold M(\Bbb Z)}\int_{\frak R} \det(Y)^s e^{-4\pi (T,Y)}\, d^*Y \\
&& \phantom{xxxx} =\sum_{T\in \frkJ(\ZZ)_{>0}//\calm(\ZZ)} \epsilon(T)^{-1} a_F(T)\overline{a_G(T)}\sum_{m\in \bold M(\Bbb Z)}\int_{m\frak R} \det(Y)^s e^{-4\pi (T,Y)}\, d^*Y \\
&& \phantom{xxxx} =\sum_{T\in \frkJ(\ZZ)_{>0}//\calm(\ZZ)} \epsilon(T)^{-1} a_F(T)\overline{a_G(T)}\int_{R_3^+(\Bbb R)} \det(Y)^s e^{-4\pi (T,Y)}\, d^*Y.
\end{eqnarray*}

We use the fact \cite[page 538]{Bai} that
$$\int_{R_3^+(\Bbb R)} \det(Y)^s e^{-2\pi (A,Y)}\, d^*Y=\det(A)^{-s}\pi^{12}(2\pi)^{-3s}\prod_{n=0}^2 \Gamma(s-4n).
$$
Hence
\begin{equation}\label{Phi-1}
\Phi(s)=\gamma(s) \sum_{T\in \frkJ(\ZZ)_{>0}//\calm(\ZZ)} \frac {a_F(T)\overline{a_G(T)}}{\epsilon(T)\det(T)^s},
\end{equation}
where $\gamma(s)=2^{-6s}\pi^{12-3s}\prod_{n=0}^2 \Gamma(s-4n)$. Now the integrand in (\ref{main}) transforms under the action of $\gamma\in \Gamma$ as 
$$F(\gamma Z)\overline{G(\gamma Z)} \det(Im(\gamma Z))^{s+9}=F(Z)\overline{G(Z)} \det(Y)^{s+9}|j(\gamma,Z)|^{-2s-18+2k}.
$$
Therefore,
\begin{eqnarray*}
&& \Phi(s)=\sum_{\gamma\in \Gamma_\infty\backslash\Gamma} \int_{\gamma\frak F_\infty} F(Z)\overline{G(Z)} \det(Y)^{s+9}\, d^*Z \\
&& \phantom{xxx} =\sum_{\gamma\in \Gamma_\infty\backslash\Gamma} \int_{\frak F_\infty} F(\gamma Z)\overline{G(\gamma Z)} \det(Im(\gamma Z))^{s+9}\, d^*Z \\
&& \phantom{xxx} = \int_{\frak F_\infty} F(Z)\overline{G(Z)} \det(Y)^{s+9}\,\left( \sum_{\gamma\in \Gamma_\infty\backslash\Gamma} 
|j(\gamma,Z)|^{-2s-18+2k} \right)\, d^*Z.
\end{eqnarray*}

Hence, 
\begin{equation}\label{Phi-2}
\Phi(s)=\int_{\frak F} F(Z)\overline{G(Z)} E(Z,s+9-k) \det(Y)^{k}\, d^*Z.
\end{equation}

By comparing (\ref{Phi-1}) and (\ref{Phi-2}), we obtain our identity.
\end{proof}

\begin{remark} One can define $R(s,F,G)$ for cusp forms of weight $k_1, k_2$ ($k_2\geq k_1$), resp. by considering the Eisenstein series
$$E(Z,s,k_2-k_1)=\sum_{\gamma\in \Gamma_\infty\backslash \Gamma} j(\gamma,Z)^{-(k_2-k_1)}|j(\gamma,Z)|^{-2s}.
$$
\end{remark}

\section{Rankin-Selberg series for the Ikeda type lift for $E_{7,3}$}
We review the Ikeda type lift of a primitive form in \cite{K-Y} and consider its Rankin-Selberg series. 
Let $k\geq 10$ be a positive integer, and for a primitive form $f\in S_{2k-8}(SL_2(\ZZ))$, let
\[f(\tau)=\sum_{m=1}^{\infty} a_f(m)\exp(2\pi \sqrt{-1} m\tau).\]
For a prime number $p$, let  
$\alpha_p$ be a complex number such that 
$a_f(p)=p^{(2k-9)/2}(\alpha_p+\alpha_p^{-1}).$ By Deligne's theorem, we have $|\alpha_p|=1$. We define the automorphic $L$-function $L(s,\pi_f)$ of the cuspidal representation $\pi_f$ attached to $f$ as 
\[L(s,\pi_f)=\prod_p \{(1-p^{-s}\alpha_p)(1-p^{-s}\alpha_p^{-1})\}^{-1}.\]
We also define the symmetric square and the symmetric cube $L$-functions $L(s,\mathrm{Sym}^2\pi_f)$ and $L(s,\mathrm{Sym}^3\pi_f)$ as 
\[L(s,\mathrm{Sym}^2\pi_f)=\prod_p\{(1-p^{-s}\alpha_p^2)(1-p^{-s}\alpha_p^{-2})(1-p^{-s})\}^{-1}\]
and
\[L(s,\mathrm{Sym}^3\pi_f)=\prod_p\{(1-p^{-s}\alpha_p^3)(1-p^{-s}\alpha_p)(1-p^{-s}\alpha_p^{-1})(1-p^{-s}\alpha_p^{-3})\}^{-1}.\]
To construct the lift in question, let us consider the local Siegel series.  Let $p$ be a prime number. For $T \in \frkJ(\QQ_p)$, let $T \sim_{\calm(\ZZ_p/p^n\ZZ_p)}  \epsilon_1 p^{a_1} \bot \epsilon_2 p^{a_2} \bot \epsilon_3 p^{a_3}$ with $a_1,a_2,a_3\in\Bbb Z\cup \{\infty\}$, $a_1\leq a_2\leq a_3$, and $\epsilon_i\in \Bbb Z_p^\times$. 
Define $\kappa_p(T)$ by $\displaystyle \kappa_p(T)=\prod_{1 \le i \le 3 \atop a_i>0} p^{a_i}$. Here we make the convention that $\kappa_p(T)=1$ if $T=O$. We note that $\kappa_p(T)$ is uniquely determined by 
$T \text{ mod } \frkJ(\ZZ_p)$. 
For $T \in \frkJ(\ZZ_p)^{\rm ns}$,  let $S_p(T)$ be the local Siegel series defined by
\[S_p(s,T)=\sum_{T' \in \frkJ(\QQ_p)/\frkJ(\ZZ_p)} {\bf e}((T,T'))\kappa_p(T')^{-s}.\]
Then, there is a polynomial $f_T^p(X)$ in $X$ such that 
\[S_p(s,T)=(1-p^{-s})(1-p^{4-s})(1-p^{8-s})f_T^p(p^{9-s}).\]
Put
\[\widetilde f_T^p(X)=X^{\ord_p(\det T)}f_T^p(X^{-2}).\]
Then it satisfies the functional equation
\begin{equation}\label{functional}
\widetilde f_T^p(X^{-1})=\widetilde f_T^p(X).
\end{equation}
For $T \in \frkJ(\ZZ)_{>0}$, put
$a_{F_f}(T)=\det(T)^{\frac {2k-9}2} \prod_{p| \det(T)} \widetilde f_T^p(\alpha_p)$, and define the Fourier series $F_f(Z)$ on $\frkT$ by
$$F_f(Z)=\sum_{T\in \frkJ(\ZZ)_{>0}} a_{F_f}(T) {\bf e}((T,Z)) \quad  (Z \in \frkT).
$$
Then, the second and the third named authors showed that
$F_f$ is a cuspidal Hecke eigenform of weight $2k$ for ${\bf G}(\ZZ)$ whose degree 56 standard $L$-function is 
\[L(s,\mathrm{Sym}^3\pi_f)\prod_{i=-4}^4 L(s+i,\pi_f)\prod_{i=-8}^8 L(s+i,\pi_f).\]

We consider the Rankin-Selberg series of $F_f$.
Recall 
$$R(s,F_f,F_f)=\sum_{T\in {\frak J}(\Bbb Z)_{>0}/\calm'(\ZZ)} \frac {|a_{F_f}(T)|^2}{\epsilon(T) \det(T)^s}.
$$Even though $R(s,F,G)$ does not have an Euler product for general $F,G$, we show that $R(s,F_f,F_f)$ has an Euler product, which enables us to reduce its computation to each $p$-adic place.

For $d \in \ZZ_p \setminus \{0\}$, put
\[\lambda_p(d,X)=\sum_{T \in \frkJ(d,\ZZ_p)/\calm'(\ZZ_p)} {\widetilde f_T^p(X)^2 \over  \beta_p(T)},\]
and
for a positive integer $d$, put
\[C(d;f)=\prod_{p<\infty} \lambda_p(d,\alpha_p).\]

\begin{theorem} \label{th.local-global-RS1}
We have
\[R(s,F_f,F_f)=c \sum_{d=1}^{\infty}C(d,f)d^{-s+2k},\]
where $c$ is a non-zero constant in Theorem \ref{th.mass-formula2}.
\end{theorem}
\begin{proof}
Let $\calg =\frkJ(\ZZ)_{>0}/\!\approx$ be the set of all genera of $\frkJ(\ZZ)_{>0}$.
We note that the Fourier coefficient $a_{F_f}(T)$ is uniquely determined by $\calg(T)$. Hence, by Theorem \ref{th.mass-formula2}, we have
\begin{align*}
R(s,F_f,F_f)&=\sum_{T \in \calg}\sum_{T' \in \calg(T)/\calm'(\ZZ)} (\det T')^{-s} {|a_{F_f}(T')|^2 \over \epsilon(T')}=\sum_{T \in \calg}\mathrm{Mass}(T) (\det T)^{-s} |a_{F_f}(T)|^2 \\
&=c\sum_{T \in \calg}  (\det T)^{2k-s}\prod_p { \widetilde f_T^p(\alpha_p)\widetilde f_T^p(\bar \alpha_p) \over \beta_p(T)}\\
&=c\sum_{d=1}^{\infty}  d^{-s+2k} \sum_{ T \in \frkJ(d,\ZZ)/\prod\calm'(\ZZ_p)} \prod_p{ \widetilde f_T^p(\alpha_p)\widetilde f_T^p(\bar \alpha_p) \over \beta_p(T)}.
\end{align*}
Since $\bar\alpha_p=\alpha_p^{-1}$, the functional equation of $\widetilde f_T^p(X)$ (\ref{functional}) implies
$\widetilde f_T^p(\alpha_p)=\widetilde f_T^p(\bar \alpha_p)$. 
Thus the assertion follows from Proposition \ref{prop.local-global-exceptional-Jordan}.
\end{proof}

For $d \in \ZZ_p^{\times}$, define a formal power series $H_p(d;X,t)$ by
\[H_p(d;X,t)=\sum_{m=0}^{\infty} \lambda_p(p^md,X)t^{m}.\]

\begin{lemma}\label{lem.independence-of-d}
$\lambda_p(d;X)$ is determined by $\ord_p(d)$.
\end{lemma}
\begin{proof}
The assertion follows from Lemma \ref{lem.Jordan-decomposition} (2).
\end{proof}

By Lemma \ref{lem.independence-of-d}, $H_p(d;X,t)$ does not depend on the choice of $d \in \ZZ_p^{\times}$, and
we write it as $H_p(X,t)$. Hence
\begin{theorem}\label{th.local-global-RS2}
We have
\[R(s,F_f,F_f)=c\prod_p H_p(\alpha_p,p^{-s+2k}).\]
\end{theorem}

\section{Explicit formula  for $\beta_p(T)$}

We give an explicit formula for the local density.
We define the local zeta function $\zeta_{\calm_p}(s)$ and $Z_{\frkJ,p}(s)$ by
\[\zeta_{\calm_p}(s) =\sum_{T \in \frkJ(\ZZ_p)/\calm(\ZZ_p)} {1 \over \beta_p(T) p^{s\, \ord_p(\det T)}}\]
and
\[Z_{\frkJ,p}(s)=\int_{\frkJ(\ZZ_p)} |\det x|^{s} |d\sigma(x)|_p.\]

\begin{proposition} \label{prop.Igusa}
We have 
\[\zeta_{\calm_p}(s)={1 \over (1-p^{-2})(1-p^{-6})(1-p^{-8})(1-p^{-12}) 
(1-p^{-s-1})(1-p^{-s-5})(1-p^{-s-9})}.\]
\end{proposition}
\begin{proof}
Let $\calh(i;\ZZ_p)=\{T \in \frkJ(\ZZ_p) \ | \ \ord_p(\det T)=i\}$.
Then we have
\begin{align*}
Z_{\frkJ,p}(s)=\sum_{i=0}^{\infty} p^{-is}  \int_{\calh(i;\ZZ_p)} |d\sigma(x)|_p. 
\end{align*}
By Theorem \ref{th.interpretation-of-local-density}, we have
\begin{align*}
\int_{\calh(i;\ZZ_p)} |d\sigma(x)|_p&=\sum_{T \in \calh(i;\ZZ_p)/\calm(\ZZ_p)} \int_{\calo_{\calm(\ZZ_p)}(T)} |d\sigma(x)|_p=(1-p^{-1})\delta_p \sum_{T \in \calh(i;\ZZ_p)/\calm(\ZZ_p)} {1 \over \beta_p(T)}. 
\end{align*}
Hence we have 
\[\zeta_{\calm_p}(s)=(1-p^{-1})^{-1}\delta_p^{-1}Z_{\frkJ,p}(s).\]
By \cite[Lemma 5]{Ig}, we have
\[Z_{\frkJ,p}(s)={(1-p^{-1})(1-p^{-5})(1-p^{-9}) \over (1-p^{-s-1})(1-p^{-s-5})(1-p^{-s-9})}.\]
This proves the proposition.
\end{proof}

\begin{corollary}
\label{cor.Igusa}
Let $T \in \frkJ(\ZZ_p)$ and suppose that $\det T \not=0$.
\begin{itemize}
\item[(1)] Suppose that $\ord_p(\det T)=0$. Then
\[\beta_p(T)=(1-p^{-2})(1-p^{-6})(1-p^{-8})(1-p^{-12}).\]
\item[(2)] Suppose that $\ord_p(\det T)=1$. Then
\[\beta_p(T)=p(1-p^{-2})(1-p^{-4})(1-p^{-6})(1-p^{-8}).\]
\end{itemize}
\end{corollary}
\begin{proof}
We have
\begin{equation}\label{Igusa}
\zeta_{\calm_p}(s)=\sum_{i=0}^{\infty} \sum_{T \in \calh(i;\ZZ_p)/\calm(\ZZ_p)} {1 \over \beta_p(T)}p^{-is}.
\end{equation}
By Lemma \ref{lem.Jordan-decomposition}, if $\ord_p(\det T)=0$, then 
$T \sim_{\calm(\ZZ_p)} 1_3$. Hence by (\ref{Igusa}) and Proposition \ref{prop.Igusa}, we have 
\[\beta_p(1_3)=(1-p^{-2})(1-p^{-6})(1-p^{-8})(1-p^{-12}).\]
Next, if $\ord_p(\det T)=1$, then 
$T \sim_{\calm(\ZZ_p)} 1_2 \bot p$. Hence again by (\ref{Igusa}) and Proposition \ref{prop.Igusa}, we have
\[{1 \over \beta_p(1_2 \bot p)}={p^{-1}+p^{-5}+p^{-9} \over (1-p^{-2})(1-p^{-6})(1-p^{-8})(1-p^{-12})}.\]
This proves that
\begin{eqnarray*}
\beta_p(1_2 \bot p)&=&p{(1-p^{-2})(1-p^{-6})(1-p^{-8})(1-p^{-12}) \over 1+p^{-4}+p^{-8}} \\
&=&p(1-p^{-2})(1-p^{-4})(1-p^{-6})(1-p^{-8}).
\end{eqnarray*}
\end{proof}

\begin{proposition} \label{prop.induction-formula-of-local-density1}
Let $T \in \frkJ(\ZZ_p)$ and suppose that $\det T \not=0$.
\begin{itemize}
\item[(1)] $\beta_p(pT)=p^{27}\beta_p(T)$.
\item[(2)] $\beta_p(T \times T)=p^{9\, \ord_p(\det T)}\beta_p(T)$.
\end{itemize}
\end{proposition}
\begin{proof}
Let $\widetilde T$ be an element of $\frkJ(\ZZ)_{>0}$ satisfying the condition in Proposition \ref{prop.approximation} (2).
Then, by Theorem \ref{th.mass-formula2}, we have 
\[\mathrm{Mass}(\widetilde T)={c (\det \widetilde T)^9 \over \prod_{q<\infty} \beta_q(\widetilde T)},\quad \mathrm{Mass}(p\widetilde T)={c (\det (p\widetilde T))^9 \over \prod_{q<\infty} \beta_q(p\widetilde T)},\quad \mathrm{Mass}(\widetilde T \times \widetilde T)={c (\det (\widetilde T \times \widetilde T) )^9 \over \prod_{q<\infty} \beta_q(\widetilde T \times \widetilde T)},\]
By definition, $\beta_q(p\widetilde T)=\beta_q(\widetilde T)$ for any $q \not=p$. Now we show for any $T\in \mathcal J(\ZZ)_{>0}$ and $m\in\Bbb Z_{>0}$, $\mathrm{Mass}(mT)=\mathrm{Mass}(T)$: Consider the definition of $\mathrm{Mass}(T)$ in (\ref{mass-def}). Note that there is a bijection between $\mathcal G(T)$ and $\mathcal G(mT)$. Then clearly, $\mathcal U_{mT}(\ZZ)=\mathcal U_T(\ZZ)$. So $\epsilon(mT)=\epsilon(T)$. This proves the result.
Hence $\mathrm{Mass}(p\widetilde T)=\mathrm{Mass}(\widetilde T)$. Therefore we have
\[\beta_p(pT)=\beta_p(p\widetilde T)=p^{27}\beta_p(\widetilde T)=p^{27}\beta_p(T).\]
This proves (1). The automorphism $g \longrightarrow g^*$ of $\calm'(\ZZ)$ induces an isomorphism from $\calu_{\widetilde T}$ to $\calu_{\widetilde T \times \widetilde T}$. Hence we also have $\mathrm{Mass}(\widetilde T \times \widetilde T)=\mathrm{Mass}(\widetilde T)$ and
$\beta_q(\widetilde T \times \widetilde T)=\beta_q(\widetilde T)$ for $q \not=p$. Therefore
\begin{align*}
\beta_p(T \times T)&=\beta_p(\widetilde T \times \widetilde T)=p^{9\, \ord_p(\det (\widetilde T \times \widetilde T))-9\, \ord_p(\det \widetilde T)}\beta_p(\widetilde T) \\
&=p^{9\, \ord_p(\det ( T \times  T))-9\, \ord_p(\det  T)}\beta_p(T).
\end{align*}
Thus the assertion (2) is proved since $\det (T \times T)=(\det T)^2$.
\end{proof}

\begin{proposition} \label{prop.induction-formula-of-local-density2}
Let $T \in \frkJ(\ZZ_p)^{\rm ns}$ such  that  $e_1(T)=0$ and $e_2(T)<e_3(T)$, and let $T' \in \frkJ(\ZZ_p)$ such that $e_i(T')=e_i(T)$ for $i=1,2$ and $e_3(T')=e_3(T)+1$. Then,
\[\beta_p(T')=p\beta_p(T).\]
\end{proposition}
\begin{proof}
For positive integers $n_2$  and $n$ such that $n_2 < n$, let
\[\cala_{n_2,n}=\{X \in \frkJ(\ZZ_p/p^{n}\ZZ_p) \ | \ e_1(X)=0, \  e_1(X \times X)=n_2, \ \det X \equiv 0 \text{ mod } p^{n_2+n} \},\]
and
\[\calb_{n_2,n}=\{X \in \frkJ(\ZZ_p/p^{n}\ZZ_p) \ | \ e_1(X)=0, \  e_1(X \times X)=n_2, \ \det X \equiv 0 \text{ mod } p^{n_2+n-1} \}.\]
We note that for $X \in \frkJ(\ZZ_p)$
\[\det (X+p^nX_1) \equiv \det X + p^{n} (X \times X,X_1) \mod p^{n+n_2},\]
and that $(X \times X,X_1) \equiv 0 \text{ mod } p^{n_2}$ if $e_1(X \times X) \ge n_2$. Therefore, 
$\cala_{n_2,n}$ and $\calb_{n_2,n}$ are well defined.
Put $e_i=e_i(T)$. Then we have
\[\# \calo_{\calm(\ZZ_p/p^{e_3+1}\ZZ_p)}(\bar T)=\# \calb_{e_2,e_3+1} -\#\cala_{e_2,e_3+1},\]
and
\[\# \calo_{\calm(\ZZ_p/p^{e_3+2}\ZZ_p)}(\bar T')=\# \calb_{e_2,e_3+2} -\#\cala_{e_2,e_3+2},\]
where $\bar T=T \text{ mod } p^{e_3+1}$ and $\bar T'=T \text{ mod } p^{e_3+2}$.
Clearly we have
\[\#\calb_{n_2,n+1}=p^{27}\cala_{n_2,n}.\]
We now prove that 
\begin{equation}\label{cala}
 \#\cala_{n_2,n+1}=p^{26}\cala_{n_2,n}.
\end{equation}
Let $X \text{ mod }  p^n \in \cala_{n_2,n}$ and put $Y=X+p^{n} X_1 \mod p^{n+1}$ with $X_1 \in \frkJ(\ZZ_p)$.
Then we have
\[\det (X+p^{n} X_1) \equiv \det X +p^{n}(X \times X, X_1) \text{ mod } p^{n_2+n+1}.\]
We have $X \times X=p^{n_2} \widetilde X$ with $\widetilde X \in \frkJ(\ZZ_p)$ such that $e_1(\widetilde X)=0$.
Then we have 
\begin{equation}\label{cala2}
 Y \in \cala_{n_2,n+1} \text{ if and only if } p^{-n_2-n}\det X +(\widetilde X \times \widetilde X, X_1) \equiv 0 \text{ mod } p.
\end{equation}
The number of $X_1 \text{ mod } p$ satisfying (\ref{cala2}) is $p^{26}$. Hence we have proved (\ref{cala}).
Therefore, 
\[\# \calo_{\calm(\ZZ_p/p^{e_3+2}\ZZ_p)}(\bar T')=p^{26}\# \calo_{\calm(\ZZ_p/p^{e_3+1}\ZZ_p)}(\bar T).\]
Thus the assertion follows from Theorem \ref{th.interpretation-of-local-density}.
\end{proof}

\begin{theorem} \label{th.explicit-formula-of-local-density}
Let $T=p^{a_1} \bot p^{a_2} \bot p^{a_3}$ with $a_1 \le a_2 \le a_3$.
\begin{itemize}
\item[(1)] Let $a_1=a_2=a_3$. Then
\[\beta_p(T)=p^{27a_1}(1-p^{-2})(1-p^{-6})(1-p^{-8})(1-p^{-12}).\]
\item[(2)] Let $a_1=a_2 <a_3$. Then
\[\beta_p(T)=p^{26a_1+a_3}(1-p^{-2})(1-p^{-4})(1-p^{-6})(1-p^{-8}).\]
\item[(3)] Let $a_1<a_2=a_3$. Then
\[\beta_p(T)=p^{17a_1+10a_3}(1-p^{-2})(1-p^{-4})(1-p^{-6})(1-p^{-8}).\]
\item[(4)] Let $a_1<a_2 <a_3$. Then
\[\beta_p(T)=p^{17a_1+9a_2+a_3}(1-p^{-2})(1-p^{-4})^2(1-p^{-6}).\]
\end{itemize}
\end{theorem}
\begin{proof} 
Put 
\begin{eqnarray*}
&& c_1=(1-p^{-2})(1-p^{-6})(1-p^{-8})(1-p^{-12}),\\
&& c_2=(1-p^{-2})(1-p^{-4})(1-p^{-6})(1-p^{-8}),\\
&& c_3=(1-p^{-2})(1-p^{-4})^2(1-p^{-6}).
\end{eqnarray*}
The assertion (1) follows from Corollary \ref{cor.Igusa} and Proposition \ref{prop.induction-formula-of-local-density1} (1).
Let $a_1 <a_3$. Then, by Proposition \ref{prop.induction-formula-of-local-density1} (3), Proposition \ref{prop.induction-formula-of-local-density2}, and Corollary \ref{cor.Igusa}, we have
\begin{align*}
\beta_p(p^{a_1}1_2 \bot p^{a_3} )=p^{27a_1} \beta_p(1_2 \bot p^{a_3-a_1})=p^{27a_1} p^{a_3-a_1-1} \beta_p(1_2 \bot p)=p^{26a_1+a_3} c_2,
\end{align*}
which proves (2). 

Let $a_1 < a_2$. We note that $p^{a_2-a_1}1_2 \bot 1=T\times T$ with $T=1_2 \bot p^{a_2-a_1}$, and $p^{a_2-a_1}1_2 \bot 1$ is $\calm(\ZZ_p)$-equivalent to $1\bot p^{a_2-a_1}1_2$. 
Hence, by Proposition \ref{prop.induction-formula-of-local-density1} and (2), we have
\begin{align*}
\beta_p(p^{a_1} \bot p^{a_2}1_2)&=p^{27a_1}\beta_p(1 \bot p^{a_2-a_1}1_2) =p^{27a_1} p^{9(a_2-a_1)}\beta_p(1_2 \bot p^{a_2-a_1})=p^{17a_1+10a_2} c_2,
\end{align*}
which proves (3). 

Let $a_1<a_2<a_3$. Then, by Proposition \ref{prop.induction-formula-of-local-density1} (1) and 
Proposition \ref{prop.induction-formula-of-local-density2}, we have 
\begin{align*}
\beta_p(p^{a_1} \bot p^{a_2} \bot p^{a_3})&=p^{27a_1}\beta_p(1 \bot p^{a_2-a_1} \bot p^{a_3-a_1})=p^{27a_1} p^{a_3-a_2-1}\beta_p(1 \bot p^{a_2-a_1} \bot p^{a_2-a_1+1}).
\end{align*}
We note that 
$$
\mathrm{diag}(1 ,p^{a_2-a_1} , p^{a_2-a_1+1}) \times \mathrm{diag}(1 ,  p^{a_2-a_1} , p^{a_2-a_1+1})=p^{a_2-a_1} \mathrm{diag}(p^{a_2-a_1 +1}, p, 1).
$$
Here $\mathrm {diag}(p^{a_2-a_1+1},p,1)\sim_{\calm(\ZZ_p)} 1\bot p\bot p^{a_2-a_1+1}$.
Hence, by Propositions \ref{prop.induction-formula-of-local-density1} and \ref{prop.induction-formula-of-local-density2}, we have
\begin{align*}
& \beta_p(1 \bot p^{a_2-a_1} \bot p^{a_2-a_1+1})=p^{-9(2a_2-2a_1+1)}\beta_p(p^{a_2-a_1}\mathrm{diag}(1, p, p^{a_2-a_1+1})) \\
& \phantom{xxxxxxxxxx} =p^{-9(2a_2-2a_1+1)} p^{27(a_2-a_1)}\beta_p(1 \bot p \bot p^{a_2-a_1+1}) 
 =p^{10a_2-10a_1-10}\beta_p(1 \bot p \bot p^{2}).
\end{align*}
Since any $T\in \mathfrak J(\Bbb Z_p)$ with ord$_p(\det(T))=3$ is $\calm(\ZZ_p)$-equivalent to $p1_3, 1_2\bot p^2$ or $1\bot p\bot p^2$,
by Proposition \ref{prop.Igusa}, we have 
\begin{align*}
&{1 \over \beta_p(p1_3)}+{1 \over \beta_p(1_2 \bot p^3)} +{1 \over \beta_p(1 \bot p \bot p^2)}\\
&=c_1^{-1}(p^{-27}+p^{-3}(1+p^{-4}+p^{-8})+p^{-11}(1+p^{-4})(1+p^{-4}+p^{-8})),
\end{align*}
and hence by (1) and (2), we have 
\[\beta_p(1 \bot p \bot p^2)=p^{11}c_3.\]
This  proves (4). 
\end{proof}
\begin{lemma}\label{lem.local-integral-at-finite-place}
 Let $T$ be an element of $\frkJ(\ZZ_p)^{\rm ns}$. Then
$$\int_{\calu_T(\ZZ_p)} |\omega_{T}|_p=|\det T|_p^9 \beta_p(T){\int_{\calm(\ZZ_p)} |dg'|_p \over \delta_p}.$$
In particular, for any prime number $p \not\in \cals$, where $\cals$ is a finite set of prime numbers in Lemma \ref{lem.normalization-of-dg}, 
we have
$$\int_{\calu_{T}(\ZZ_p)} |\omega_{T}|_p=|\det T|_p^9 \beta_p(T).$$

\end{lemma}
\begin{proof}
As in the proof of Theorem 
\ref{th.interpretation-of-local-density}, we have
$$\int_{\calm(\ZZ_p)}|dg|_p=\int_{\calm(\ZZ_p)\cdot T}|\eta(g \cdot T)|_p\int_{\calu_T(\ZZ_p)}|\omega_{T}|_p$$
and
$$\int_{\calm(\ZZ_p)\cdot T}|\eta(g \cdot T)|_p=\int_{\calo_{\calm(\ZZ_p)}(T)} |\det T|_p^{-9} |d\sigma(x)|_p.$$
Thus the assertion follows from Lemma \ref{lem.normalization-of-dg} and Definition \ref{def.local-density}.
\end{proof}

For our later purpose, we show the following:
\begin{proposition} \label{prop.constant-at-infinity}
Let $c$ be the constant in Theorem \ref{th.mass-formula2}. Then
$$c={5!\cdot7! \cdot11! \over (2\pi)^{28}}.
$$
Hence
$$c \zeta(2)\zeta(6)\zeta(8)\zeta(12) ={691 \over 2^{15}\cdot 3^6 \cdot 5^2 \cdot 7^2 \cdot 13}\in \Bbb Q.
$$
\end{proposition}
\begin{proof}
By \cite[page 273]{Gr}, $\calu_{1_3}$ is an integral model of ${\bf U}_{1_3}$. Let $\omega$ be a differential form which generates the rank one module of differential of the top degree on $\calu_{1_3}$ over $\ZZ$ (cf. \cite[page 268]{Gr}). 
Then, by \cite[pages 268-269]{Gr} and \cite[Table 5.2]{Gr}, we have
$$\int_{\calu_{1_3}(\ZZ_p)}|\omega|_p=(1-p^{-2})(1-p^{-6})(1-p^{-8})(1-p^{-12}),
$$
for any prime number $p$, and
$$\int_{\calu_{1_3}(\RR)}|\omega|_\infty={(2\pi)^{28} \over 5!\cdot7! \cdot11!}.
$$
On the other hand, by Corollary \ref{cor.Igusa} and Lemma \ref{lem.local-integral-at-finite-place}, we have
$$\int_{\calu_{1_3}(\ZZ_p)}|\omega|_p{\int_{\calm(\ZZ_p)} |dg'|_p \over \delta_p}=\int_{\calu_{1_3}(\ZZ_p)}|\omega_{1_3}|_p,
$$
for any prime number $p$. Hence $\omega_{1_3}=\pm \omega$ and therefore we have
$$\int_{\calu_{1_3}(\RR)}|\omega|_\infty=\int_{\calu_{1_3}(\RR)}|\omega_{1_3}|_\infty d_0^{-1}=c_0^{-1}d_0^{-1}=c^{-1}.
$$
This proves the assertion for $c$. The remaining assertion follows from $\zeta(2)=\frac {\pi^2}6$, $\zeta(6)=\frac {\pi^6}{3^3\cdot 5\cdot 7}$, $\zeta(8)=\frac {\pi^8}{2\cdot 3^3\cdot 5^2\cdot 7}$, $\zeta(12)=\frac {691 \pi^{12}}{3^6\cdot 5^3\cdot 7^2\cdot 11\cdot 13}$.
\end{proof}

We remark that $\beta_p(T)$ is an analogue of the local density of a quadratic form. 
To explain this, first we have the following lemma.
\begin{lemma}\label{lem.counting-principle}
Let $n$ be a positive integer. For $T \in \frkJ(\ZZ_p/p^n\ZZ_p)$, let  $\calm(\ZZ_p/p^n\ZZ_p)\cdot T$ be the orbit of $T$ under $\calm(\ZZ_p/p^n\ZZ_p)$. Then
\[\#(\calo_{\calm(\ZZ_p/p^n\ZZ_p)}(T)) =\#\calm(\ZZ_p/p^n\ZZ_p)/\#\calu_T(\ZZ_p/p^n\ZZ_p).\]
\end{lemma}
\begin{proof}
The mapping
\[\calm(\ZZ_p/p^n\ZZ_p) \ni g \mapsto g\cdot T \in \calm(\ZZ_p/p^n\ZZ_p)\cdot T\]
is surjective and for $g_1,g_2 \in \calm(\ZZ_p/p^n\ZZ_p)$, we have
\[g_1T=g_2T \text{ if and only if } g_1^{-1}g_2 \in \calu_T(\ZZ_p/p^n\ZZ_p).\]
Thus the assertion holds.
\end{proof}
By Lemma \ref{lem.order-of-GE_6}, Theorem \ref{th.interpretation-of-local-density}, and Lemma \ref{lem.counting-principle},
for $T \in \frkJ(\ZZ_p)^{\rm ns}$, we have 
\[\beta_p(T)=\lim_{n \rightarrow \infty} p^{-52n}\#
\calu_T(\ZZ_p/p^n\ZZ_p).\]
This is just an analogue of the local density in the theory of quadratic forms. (See, for example, \cite{KK}.)

\section{Explicit formula for  $H_p(X,t)$} 

We first rewrite the formula of the Siegel series due to Karel \cite{Kar3}.

\begin{theorem} \label{th.explcit-formula-of-Siegel-series}
Let $T=p^{m_1} \bot p^{m_1+m_2} \bot p^{m_1+m_3}$ with $0 \le m_1, 0  \le m_2 \le m_3$.
Then
\begin{align*}
& f_T^p(X)={1 \over (1-X)(1-p^{4}X)(1-p^8X)}+ {X^{m_2+m_3+3m_1} \over (1-X^{-1})(1-p^{4}X^{-1})(1-p^8X^{-1})} \\
&-{p^{8m_1+8}X^{m_1+1} \over (1-X)(1-p^{4}X)(1-p^8X)}- {p^{8m_1+8}X^{2m_1+m_2+m_3-1} \over (1-X^{-1})(1-p^4X^{-1})(1-p^8X^{-1})} \\
&-{p^{8m_1+4(m_2+1)}X^{m_1+m_2+1} \over (1-X)^2(1-p^4X)}-{p^{8m_1+4(m_2+1)}X^{2m_1+m_3-1} \over (1-X^{-1})^2(1-p^4X^{-1})}\\
&-{p^{8m_1+4m_2}X^{m_1+m_3+1} \over (1-X)^2(1-p^{-4}X)}-{p^{8m_1+4m_2}X^{2m_1+m_2-1} \over (1-X^{-1})^2(1-p^{-4}X^{-1})}.
\end{align*}
\end{theorem}
\begin{proof}
Let $T_0=1 \bot p^{m_2} \bot p^{m_3}$. Then, by \cite[page 553, line 8 below]{Kar3}, we have 
\begin{align*}
f_T^p(X)=f_{T_0}^p(X)(C_0(X^{-1})X^{3m_1} +C_1(X^{-1})p^{8m_1}X^{m_1}+C_1(X)p^{8m_1}X^{m_1}+C_0(X)),
\end{align*}
where
\begin{align*}
&C_0(X)={1 \over (1-X)(1-p^4X)(1-p^8 X)f_{T_0}(X)} \\
&C_1(X)={-{1+(1+p^4)X \over 1-p^8X} +f_{T_0}^p(X)-X^2f_{p^{-1}T_0}^p(X) \over (1-X)(1-p^4X)f_{T_0}^p(X)}.
\end{align*}
Hence we have
\begin{align*}
& f_{T}^p(X)={1 \over (1-X)(1-p^4X)(1-p^8X)} + {X^{3m_1+m_2+m_3} \over (1-X^{-1})(1-p^4X^{-1})(1-p^8X^{-1})} \\
&-{p^{8m_1}X^{m_1}(1+(1+p^4)X) \over (1-X)(1-p^{-4}X)(1-p^8X)}
+{p^{8m_1}X^{m_1} f_{T_0}^p(X) \over (1-X)(1-p^{-4}X)}\\
&-{p^{8m_1} X^{2m_1+m_2+m_3}(1+(1+p^4)X^{-1}) \over (1-X^{-1})(1-p^{-4}X^{-1})(1-p^8X^{-1})} + {p^{8m_1}X^{2m_1}f_{T_0}^p(X) \over (1-X^{-1})(1-p^{-4}X^{-1})}\\
&-{p^{8m_1}X^{m_1+2} f_{p^{-1}T_0}(X) \over (1-X)(1-p^{-4}X)}-{p^{8m_1}X^{2m_1} f_{p^{-1}T_0}^p(X_0) \over (1-X^{-1})(1-p^{-4}X^{-1})}.
\end{align*}
By \cite[page 553, line 10 below]{Kar3}, we have
\begin{align*}
&f_{T_0}^p(X)=\sum_{k=0}^{m_2} (p^4X)^k {1-X^{m_2+m_3+1-2k} \over 1-X}\\
&\phantom{xxxxx}={1-(p^4X)^{m_2+1} \over (1-X)(1-p^4X)}-{X^{m_2+m_3+1}(1-(p^4X^{-1})^{m_2+1}) \over (1-X)(1-p^4X^{-1})}.
\end{align*}
We also have 
\begin{align*}
&f_{p^{-1}T_0}^p(X)={1-(p^4X)^{m_2} \over (1-X)(1-p^4X)}-{X^{m_2+m_3-1}(1-(p^4X^{-1})^{m_2}) \over (1-X)(1-p^4X^{-1})}.
\end{align*}
This agrees with the convention that $f_{p^{-1}T_0}^p(X)=0$ if $m_2=0$. Hence we have
\begin{align*}
& f_{T}^p(X)={1 \over (1-X)(1-p^4X)(1-p^8X)} + {X^{3m_1+m_2+m_3} \over (1-X^{-1})(1-p^4X^{-1})(1-p^8X^{-1})} \\
&-{p^{8m_1}X^{m_1}(1+(1+p^4)X) \over (1-X)(1-p^{-4}X)(1-p^8X)}\\
& +{p^{8m_1}X^{m_1}  \over (1-X)(1-p^{-4}X)} \Bigl({1-(p^4X)^{m_2+1} \over (1-X)(1-p^4X)}-{X^{m_2+m_3+1}(1-(p^4X^{-1})^{m_2+1}) \over (1-X)(1-p^4X^{-1})}\Bigr) \\
&-{p^{8m_1} X^{2m_1+m_2+m_3}(1+(1+p^4)X^{-1}) \over (1-X^{-1})(1-p^{-4}X^{-1})(1-p^8X^{-1})} \\
& + {p^{8m_1}X^{2m_1} \over (1-X^{-1})(1-p^{-4}X^{-1})}  \Bigl({1-(p^4X)^{m_2+1} \over (1-X)(1-p^4X)}-{X^{m_2+m_3+1}(1-(p^4X^{-1})^{m_2+1}) \over (1-X)(1-p^4X^{-1})} \Bigr)\\
&-{p^{8m_1}X^{m_1+2} \over (1-X)(1-p^{-4}X)}
\Bigl({1-(p^4X)^{m_2} \over (1-X)(1-p^4X)}-{X^{m_2+m_3-1}(1-(p^4X^{-1})^{m_2}) \over (1-X)(1-p^4X^{-1})}\Bigr)\\
&-{p^{8m_1}X^{2m_1} \over (1-X^{-1})(1-p^{-4}X^{-1})}\Bigl({1-(p^4X)^{m_2} \over (1-X)(1-p^4X)}-{X^{m_2+m_3-1}(1-(p^4X^{-1})^{m_2}) \over (1-X)(1-p^4X^{-1})} \Bigr).
\end{align*}
By simple computation, we have
\begin{align*}
&-{p^{8m_1}X^{m_1}(1+(1+p^4)X) \over (1-X)(1-p^{-4}X)(1-p^8X)} +{p^{8m_1}X^{m_1}  \over (1-X)(1-p^{-4}X)(1-X)(1-p^4X)} \\
& -{p^{8m_1}X^{m_1+2} \over (1-X)(1-p^{-4}X)(1-X)(1-p^4X)} =-{p^{8m_1+8}X^{m_1+1} \over (1-X)(1-p^4X)(1-p^8X)},
\end{align*}
\begin{align*}
&-{p^{8m_1}X^{2m_1+m_2+m_3}(1+(1+p^4)X^{-1}) \over (1-X^{-1})(1-p^{-4}X^{-1})(1-p^8X^{-1})} -{p^{8m_1}X^{2m_1+m_2+m_3+1}  \over (1-X^{-1})(1-p^{-4}X^{-1})(1-X)(1-p^4X^{-1})} \\
& +{p^{8m_1}X^{2m_1+m_2+m_3-1} \over (1-X^{-1})(1-p^{-4}X^{-1})(1-X)(1-p^4X^{-1})} =-{p^{8m_1+8}X^{2m_1+m_2+m_3-1} \over (1-X^{-1})(1-p^4X^{-1})(1-p^8X^{-1})},
\end{align*}
and
\begin{align*}
-{p^{8m_1}X^{m_1}(p^4X)^{m_2+1} \over (1-X)^2(1-p^{-4}X)(1-p^4X)}+{p^{8m_1}X^{m_1+2} (p^4X)^{m_2} \over (1-X)^2(1-p^{-4}X)(1-p^4X)}=-{p^{8m_1+4m_2+4}X^{m_1+m_2+1} \over (1-X)^2(1-p^4X)},
\end{align*}
\begin{align*}
&-{p^{8m_1}X^{2m_1}(p^4X)^{m_2+1} \over (1-X^{-1})(1-p^{-4}X^{-1})(1-X)(1-p^4X)}+{p^{8m_1}X^{2m_1} (p^4X)^{m_2} \over (1-X^{-1})(1-p^{-4}X^{-1})(1-X)(1-p^4X)}\\
&=-{p^{8m_1+4m_2}X^{2m_1+m_2-1} \over (1-X^{-1})^2(1-p^{-4}X^{-1})},
\end{align*}
\begin{align*}
{p^{8m_1}X^{m_1+m_2+m_3+1} (p^4X^{-1})^{m_2+1}\over (1-X)^2(1-p^{-4}X)(1-p^4X^{-1})}-{p^{8m_1}X^{m_1+m_2+m_3+1}(p^4X^{-1})^{m_2} \over (1-X)^2(1-p^{-4}X)(1-p^4X^{-1})}= -{p^{8m_1+4m_2} X^{m_1+m_3+1} \over (1-X)^2(1-p^{-4}X)},
\end{align*}
\begin{align*}
&{p^{8m_1}X^{2m_1+m_2+m_3+1} (p^4X^{-1})^{m_2+1} \over (1-X^{-1})(1-p^{-4}X^{-1})(1-X)(1-p^4X^{-1})}-{p^{8m_1}X^{2m_1+m_2+m_3-1}(p^4X^{-1})^{m_2}  \over (1-X^{-1})(1-p^{-4}X^{-1})(1-X)(1-p^4X^{-1})}\\
&=-{p^{8m_1+4m_2+4} X^{2m_1+m_3-1} \over (1-X^{-1})^2(1-p^4X^{-1})}.
\end{align*}
This proves the theorem.
\end{proof}
Since $\widetilde f_T^p(X)=X^{-\mathrm {ord}(\det T)}f_T^p(X^2)$, we have
\begin{corollary}
\label{cor.explicit-formula-of-Laurent-series}
Let $T=p^{m_1} \bot p^{m_1+m_2} \bot p^{m_1+m_3}$ with $0 \le m_1, 0  \le m_2 \le m_3$.
Then
\begin{align*}
&\widetilde f_T^p(X)={X^{-m_2-m_3-3m_1} \over (1-X^2)(1-p^{4}X^2)(1-p^8X^2)}+ {X^{m_2+m_3+3m_1} \over (1-X^{-2})(1-p^{4}X^{-2})(1-p^8X^{-2})} \\
&-{p^{8m_1+8}X^{-m_1-m_2-m_3+2} \over (1-X^2)(1-p^{4}X^2)(1-p^8X^2)}- {p^{8m_1+8}X^{m_1+m_2+m_3-2} \over (1-X^{-2})(1-p^{4}X^{-2})(1-p^8X^{-2})} \\
&-{p^{8m_1+4(m_2+1)}X^{-m_3+m_2-m_1+2} \over (1-X^2)^2(1-p^4X^2)}-{p^{8m_1+4(m_2+1)}X^{m_3-m_2+m_1-2} \over (1-X^{-2})^2(1-p^4X^{-2})}\\
&-{p^{8m_1+4m_2}X^{m_3-m_2-m_1+2} \over (1-X^2)^2(1-p^{-4}X^2)}-{p^{8m_1+4m_2}X^{-m_3+m_2+m_1-2} \over (1-X^{-2})^2(1-p^{-4}X^{-2})}.
\end{align*}
\end{corollary}

\begin{lemma}\label{lem.formal-power-series}
For a variable $A,B,C$, define a formal power series $P(A,B,C,t)$ in $t$ as
\begin{align*}
P(A,B,C,t)=\sum_{m_1 \ge 0,\,  0\le m_2 \le m_3} {t^{3m_1+m_2+m_3} A^{m_1}B^{m_2}C^{m_3} \over \beta_p(p^{m_1} \bot p^{m_1+m_2} \bot p^{m_1+m_3})}.
\end{align*}
Then
\begin{align*}
&P(A,B,C,t)=(1-p^{-2})(1-p^{-6})^{-1}(1-p^{-8})^{-1}(1-p^{-12})^{-1} \\
&\times {1+(p^{-5}+p^{-9})tC+(p^{-14}+p^{-18})t^2BC+p^{-23}t^3BC^2 \over (1-p^{-27}At^3)(1-p^{-10}BCt^2)(1-p^{-1}Ct)}.
\end{align*}
\end{lemma}

\begin{proof}
We have 
\begin{align*}
&P(A,B,C,t)=\sum_{m_1=0}^{\infty} { t^{3m_1}A^{m_1} \over \beta_p(p^{m_1} \bot p^{m_1} \bot p^{m_1})} +\sum_{m_1=0}^{\infty}\sum_{m_3=0}^{\infty} {t^{3m_1+m_3+1}A^{m_1}C^{m_3+1} \over \beta_p(p^{m_1} \bot p^{m_1} \bot p^{m_1+m_3+1})}\\
&+\sum_{m_1=0}^{\infty} \sum_{m_2=0}^{\infty} {t^{3m_1+2m_2+2}A^{m_1}B^{m_2+1}C^{m_2+1} \over \beta_p(p^{m_1} \bot p^{m_1+m_2+1} \bot p^{m_1+m_2+1})}\\
&+\sum_{m_1=0}^{\infty}\sum_{m_2=0}^{\infty}\sum_{m_3=0}^{\infty} {t^{3m_1+2m_2+m_3+3}A^{m_1}B^{m_1+m_2+1}C^{m_1+m_2+m_3+2} \over \beta_p(p^{m_1} \bot p^{m_1+m_2+1} \bot p^{m_1+m_2+m_3+2})}.
\end{align*}
Put $d=(1-p^{-2})(1-p^{-6})(1-p^{-8})(1-p^{-12})$. and 
$\widetilde P(A,B,C,t)=d P(A,B,C,t)$. Then by Theorem \ref{th.explicit-formula-of-local-density},
we have 
\begin{align*}
&\widetilde P(A,B,C,t)=\sum_{m_1=0}^{\infty} t^{3m_1}A^{m_1}p^{-27m_1} \\
&+(1+p^{-4}+p^{-8}) \sum_{m_1=0}^{\infty} \sum_{m_3=0}^{\infty} t^{3m_1+m_3+1} A^{m_1}C^{m_3+1}p^{-27m_1-m_3-1}\\
&+(1+p^{-4}+p^{-8})\sum_{m_1=0}^{\infty} \sum_{m_2=0}^{\infty}t^{3m_1+2m_2+2}A^{m_1}B^{m_2+1}C^{m_2+1} p^{-27m_1-10m_2-10}\\
&+(1+p^{-4})(1+p^{-4}+p^{-8})\sum_{m_1=0}^{\infty}\sum_{m_2=0}^{\infty}\sum_{m_3=0}^{\infty}t^{3m_1+2m_2+m_3+3}A^{m_1}B^{m_1+m_2+1}C^{m_1+m_2+m_3+2}
\\
& \times p^{-27m_1-10m_2-m_3-11}\\
&={1 \over 1-t^3Ap^{-27}} +{(1+p^{-4}+p^{-8})Cp^{-1}t \over (1-t^3Ap^{-27})(1-tCp^{-1})}\\
&+{(1+p^{-4}+p^{-8})t^2BCp^{-10} \over (1-t^3Ap^{-27})(1-t^2p^{-10}BC)}
+{(1+p^{-4})(1+p^{-4}+p^{-8})t^3BC^2p^{-11}\over (1-t^3Ap^{-27})(1-t^2p^{-10}BC)(1-tCp^{-1})}.
\end{align*}
Thus the assertion can be proved by simple computation.
\end{proof}

\begin{theorem} \label{th.explicit-local-RS}
We have
\begin{align*}
H_p(X,t) &=\{(1-p^{-2})(1-p^{-6})(1-p^{-8})(1-p^{-12})\}^{-1}\\
&\times  {(1-p^{-14}t^2)(1+p^{-5}t)(1+p^{-9}t) \over 1-p^{-1}t}\\
&\times {1 \over \prod_{i=1}^3 (1-p^{-4i+3}t)(1-p^{-4i+3}X^{-2}t)(1-p^{-4i+3}X^2t)}.
\end{align*}
\end{theorem}
\begin{proof}
For $i=1,2,3,4$ put
\begin{align*}
&A_1(X)=\{(1-X^2)(1-p^{4}X^2)(1-p^8X^2)\}^{-1} \\
&A_2(X)=-p^8X^2\{(1-X^2)(1-p^{4}X^2)(1-p^8X^2)\}^{-1} \\
&A_3(X)=-p^4X^2\{(1-X^2)^2(1-p^4X^2)^{-1} \\
&A_4(X)=-X^2\{(1-X^2)^2(1-p^{-4}X^2)\}^{-1} ,
\end{align*}
and for $i=5,6,7,8$ put $A_i(X)=A_{i-4}(X^{-1})$.
For $i=1,2,3,4$ we also define $X_i=X_i(X),Y_i=Y_i(X),Z_i=Z_i(X)$ as
\begin{align*}
& X_1=X^{-3},X_2=X_3=X_4=p^8X^{-1} \\
& Y_1=Y_2=X^{-1}, Y_3=p^4X, Y_4=p^4X^{-1} \\
& Z_1=Z_2=Z_3=X^{-1}, Z_4=X,
\end{align*}
and for $i=5,6,7,8$ put $X_i(X)=X_{i-4}(X^{-1}), Y_i(X)=Y_{i-4}(X^{-1}), Z_i(X)=Z_{i-4}(X^{-1})$. Then, by Corollary \ref{cor.explicit-formula-of-Laurent-series}, we have
\begin{align*}
&H_p(X,t)=\sum_{1 \le i,j \le 8} A_i(X)A_j(X)P(X_i(X)X_j(X),Y_i(X)Y_j(X),Z_i(X)Z_j(X),t).
\end{align*}
Thus the assertion can be proved by Lemma \ref{lem.formal-power-series} with the aid of Mathematica.
\end{proof}

\section{Explicit formula for the Rankin-Selberg series}
Let $F_f$ be the Ikeda type lift of $f$ for $E_{7,3}$ in Section 5. 
By Theorems \ref{th.explicit-local-RS} and \ref{th.local-global-RS2}, we obtain
\begin{theorem} \label{th.explicit-RS}
\begin{align*}
R(s,F_f,F_f) &=c \, \zeta(2)\zeta(6)\zeta(8)\zeta(12) \prod_{i=1}^3 \zeta(2s-4k+4i+6)^{-1} \\
&\times \prod_{i=1}^3 \zeta(s-2k+4i-3)L(s-2k+4i-3,{\rm Sym}^2\pi_f).
\end{align*}
\end{theorem}

By the residue formula in Theorem \ref{th.Rankin-Selberg}, we have the following period relation

\[\langle F_f, F_f \rangle =\gamma_{k} \pi^{-6k-3}\prod_{i=1}^3 L(4i-3,{\rm Sym}^2\pi_f),
\]
with $\gamma_k=c \zeta(2)\zeta(6)\zeta(8)\zeta(12) 2^{-12k+22}\cdot 3^3\cdot 5 \cdot(2k-1)!(2k-5)!(2k-9)!.$
By Proposition \ref{prop.constant-at-infinity},
 we have
$$\gamma_k={ (2k-1)!(2k-5)!(2k-9)! \ 691 \over 2^{12k-7} \cdot 3^3 \cdot 5 \cdot 7^2\cdot 13}\in \Bbb Q.
$$
This proves Theorem \ref{th.period-relation}.

\end{document}